\documentclass[10pt,reqno]{amsart}
\usepackage{amsmath,amsopn,amssymb,amsthm}
\usepackage[cp1251]{inputenc}
\usepackage[english]{babel}
\usepackage[pagebackref,breaklinks=true,colorlinks=true,linkcolor=blue,citecolor=blue,urlcolor=blue]{hyperref}
\usepackage{graphicx}%
\usepackage{color}

\newtheorem{theorem}{Theorem}[section]

\newtheorem{corollary}[theorem]{Corollary}

\newtheorem{definition}[theorem]{Definition}
\newtheorem{example}[theorem]{Example}

\newtheorem{lemma}[theorem]{Lemma}

\newtheorem{question}[theorem]{Question}

\newtheorem{remark}[theorem]{Remark}

\renewenvironment{proof}[1][Proof]{\noindent\textbf{#1.} }{\ \rule{0.5em}{0.5em}}

\begin{document}
\title[Minkowski norm and Hessian isometry induced by an isoparametric foliation]{Minkowski norm and Hessian isometry induced by an isoparametric foliation on the unit sphere}
\author{Ming Xu}

\address{Ming Xu \newline
School of Mathematical Sciences,
Capital Normal University,
Beijing 100048,
P. R. China}
\email{mgmgmgxu@163.com}

\date{}

\begin{abstract}
Let $M_t$ be an isoparametric foliation on the unit sphere $(S^{n-1}(1),g^{\mathrm{st}})$ with $d$ principal curvatures. Using the spherical coordinates
induced by $M_t$, we construct a Minkowski norm with the presentation $F=r\sqrt{2f(t)}$, which generalizes the notions of $(\alpha,\beta)$-norm and $(\alpha_1,\alpha_2)$-norm. Using the technique of spherical local frame, we give
an exact and explicit answer for the question when $F=r\sqrt{2f(t)}$
really defines a Minkowski norm. Using the similar technique, we study the Hessian isometry $\Phi$ between two Minkowski norms induced by $M_t$, which preserves
the orientation and fixes the spherical $\xi$-coordinates. There are
two ways to describe this $\Phi$, either by a system of ODEs,
or by its restriction to
any normal plane for $M_t$, which is then reduced to a Hessian isometry between Minkowski norms on $\mathbb{R}^2$ satisfying certain symmetry and (d)-properties. When $d>2$, we prove this $\Phi$ can be obtained by gluing positive scalar multiplications and compositions between the Legendre transformation and positive scalar multiplications, so it must satisfy the (d)-property for any orthogonal decomposition $\mathbb{R}^n=\mathbf{V}'+\mathbf{V}''$,
i.e.,
for any nonzero $x=x'+x''$ and $\Phi(x)=\overline{x}=\overline{x}'+\overline{x}''$,
with $x',\overline{x}'\in\mathbf{V}'$ and $x'',\overline{x}''\in\mathbf{V}''$,
we have $g_x^{F_1}(x'',x)=g_{\overline{x}}^{F_2}(\overline{x}'',\overline{x})
$.
As byproducts, we prove the following results. On the indicatrix $(S_F,g)$, where $F$ is a Minkowski norm induced by $M_t$ and $g$ is the Hessian metric, the foliation $N_t=S_F\cap \mathbb{R}_{>0}M_0$ is isoparametric. Laugwitz Conjecture is valid for a Minkowski norm $F$ induced by $M_t$, i.e., if its Hessian metric $g$ is flat on $\mathbb{R}^n\backslash\{0\}$ with $n>2$, then $F$ is Euclidean.


\end{abstract}\maketitle

\section{Introduction}

The classification of {\it isoparametric foliations} on the unit sphere $(S^{n-1}(1),g^{\mathrm{st}})$ (if not otherwise specified, we will
always assume $n>2$)
has been one of the most important geometric problems \cite{Ya1982}, with a history of eighty years since the time of E. Cartan \cite{Ca1938,Ca1939}. There were many remarkable progress \cite{Ch2011,Ch2013,CCJ2007,DN1985,HL1971,Mi2013,St1999}, and recently it was
completely solved by Q.S. Chi \cite{Ch2020}. Meanwhile, researchers are eager to find applications and generalizations of this theory in geometry and topology. For example, its applications in Riemannian
geometry and differential topology are concerned in \cite{GT2012,MO2014,QT2015,TXY2014,TY2013}.
Its generalization, the equifocal hypersurface, is studied in \cite{GQ2014,Ta1998,TT1995}. Its generalization to
Finsler geometry is studied in \cite{HDY-preprint,HYS2016,HYS2017,Xu2018,XMYZ2020}.
More references can be found in the survey papers \cite{Ge2019,QT2014,Th2000,TY2018}.

In this paper, we consider how to generalize and apply the isoparametric foliation on the unit spheres to the Hessian geometry \cite{Sh2007}
for Minkowski norms. This work is inspired by the
recent cowork \cite{XM2020} with V. Matveev, which implies
the interesting connections to the study of Laugwitz conjecture \cite{La1965}
in convex geometry \cite{Sc2013}
and Landsberg Unicorn Conjecture \cite{Ma1996,Sh2009}
 in Finsler geometry \cite{BCS2000}.

In this paper, we only consider smooth and strongly convex Minkowski norms on finite dimensional real vector spaces \cite{BCS2000}. For example, a {\it Minkowski norm} on $\mathbb{R}^n$ with $n\geq2$
is a continuous function $F:\mathbb{R}^n\rightarrow\mathbb{R}_{\geq0}$ which is
positive and smooth on $\mathbb{R}^n\backslash\{0\}$, and satisfies
the positive 1-homogeneity and the strong convexity (see \cite{BCS2000} or Section \ref{subsection-3-1}). Then the Hessian of $E=\tfrac12F^2$, is positive
definite at each nonzero $x$, which defines a Riemannian metric $g={\rm d}^2E$ on $\mathbb{R}^n\backslash\{0\}$. For simplicity, we call it the {\it Hessian metric} of $F$.

Since the Minkowski norm $F$ is one-to-one
determined by its {\it indicatrix} $S_F=\{x\in\mathbb{R}^n| F(x)=1\}$.
The geometric properties of the Hessian metric $g$ or its restriction to $S_F$ help us
understand the convexity of the domain enclosed by $S_F$. See \cite{Sc2013,Sh2007}
for more discussion on the relation between Hessian geometry and Convex geometry.

Notice that this is only one important
model in more general Hessian geometry. Hessian geometers
have many other sources for the function $E$ to construct the
metric \cite{GD1979,La1965,LSZ1993,Sh2013}, toric K\"{a}hler geometry,  infinite  dimensional integrable system of hydrodynamic type, affine geometry of hypersurfaces, information geometry, etc..
 More involved discussion for Hessian geometry can be found in \cite{Sh2007} and references therein.

Now we come back to an isoparametric foliation $M_t$ on the unit sphere $(S^{n-1}(1),g^{\mathrm{st}})\subset \mathbb{R}^n$. Here we parametrize $M_t$ such that $t=\mathrm{dist}_{S^{n-1}(1)}(M_t,M_0)\in[0,\tfrac{\pi}d]$, where $M_0$ and $M_{\pi/d}$ are the two focal submanifolds, and $d\in\{1,2,3,4,6\}$ is the number of principal curvatures for each $M_t$ with $t\in(0,\tfrac{\pi}d)$ in $(S^{n-1}(1),g^{\mathrm{st}})$ \cite{Mu1980}. Associated with
$M_t$, we can define the (generalized)  {\it spherical coordinates} $(r,t,\xi)\in\mathbb{R}_{>0}\times(0,\tfrac{\pi}d)\times M_{\pi/2d}$,
i.e., $x=(r,t,\xi)$ when $|x|=r$, $x/|x|\in M_t$ and there exists a normal geodesic segment in $(S^{n-1}(1),g^{\mathrm{st}})$ for this foliation, which connects $x/|x|$ to $\xi$ without passing the focal submanifolds. Further more, we introduce
{\it spherical local frame} induced by $M_t$ (see Section \ref{subsection-2-5}), with which the standard flat metric $g^{\mathrm{st}}$ on $\mathbb{R}^n\backslash\{0\}$ and its Levi-Civita connection can be explicitly calculated.

We can use the foliation $M_t$ to define a
Minkowski norm $F$ on $\mathbb{R}^n$, such that the restriction of $F$ to each $M_t$ is a constant function. We will simply call it a {\it Minkowski norm induced by $M_t$}. When $d=1$ or $2$,
the induced $F$ admits a linear $SO(n-1)$- or $O(k)\times O(n-k)$-invariancy, and is called an $(\alpha,\beta)$-norm or $(\alpha_1,\alpha_2)$-norm in some literature \cite{CS2005,DX2016}. These norms have attracted many attentions of Finsler geometers
\cite{HM2019,Ma1992}. However, the induced Minkowski norms when $d>2$
have been rarely studied.

Using the spherical $r$- and $t$-coordinates, the induced Minkowski norm $F$ can be presented as
$F=r\sqrt{2f(t)}$.
A natural and important question is the following:
\begin{question}\label{question-1}
When does $F=r\sqrt{2f(t)}$ define a Minkowski norm induced by $M_t$?
\end{question}

Notice that, besides the issue of strong convexity, the smoothness of $F=r\sqrt{2f(t)}$ at $\mathbb{R}_{>0}M_0$ and $\mathbb{R}_{>0}M_{\pi/d}$ is
also subtle and crucial. We use the spherical local frame to calculate
the Hessian of $E=\tfrac12F^2=r^2f(t)$ as in \cite{XM2020}, and then
completely answer Question \ref{question-1} by the following theorem.
\medskip

\noindent{\bf Theorem A.}\ \ {\it The spherical coordinates presentation
$F=r\sqrt{2f(t)}$ defines a Minkowski norm induced by $M_t$ if and only if $f(t)$ can be extended to a positive smooth $D_{2d}$-invariant
function on $\mathbb{R}/(2\mathbb{Z}\pi)$ which satisfies
\begin{equation*}
2f(t)\tfrac{{\rm d}^2}{{\rm d}t^2}f(t)-\left(\tfrac{{\rm d}}{{\rm d}t}f(t)\right)^2+4f(t)^2>0
\end{equation*}
everywhere, i.e., the polar coordinates presentation $\overline{F}=r\sqrt{2f(t)}$ defines a $D_{2d}$-invariant Minkowski norm on $\mathbb{R}^2$.}\medskip

Here $\mathbb{R}^2$ can be identified with any normal plane $\mathbf{V}$ for $M_t$ (i.e., $\mathbf{V}\cap S^{n-1}(1)$ is
a normal geodesic for $M_t$), and
$D_{2d}$ is the group $\mathbb{Z}_2$ when $d=1$ and the dihedral group when $d>1$, which can be interpreted as a Weyl group. See Section \ref{subsection-2-4} for its explicit description and its action on $\mathbb{R}^2$ or $\mathbb{R}/(2\mathbb{Z}\pi)$.

Theorem A is a reformulation of Theorem \ref{main-thm-1}. Its
direct corollaries, Corollary \ref{cor-1} and Corollary \ref{cor-2}, where we take $d=1$ and $2$, reprove some known results  for Minkowski norms of $(\alpha,\beta)$-
and $(\alpha_1,\alpha_2)$-types  \cite{Cr2020,DX2016}.

Let $F=r\sqrt{2f(t)}$ be a Minkowski norm induced by $M_t$. Then
on its indicatrix $S_F$, there is a foliation $N_t=S_F\cap\mathbb{R}_{>0}M_t$ induced by $M_t$. Using the technique of spherical local frame again, we prove the following theorem (see Theorem \ref{main-thm-2}).
\medskip

\noindent
{\bf Theorem B.}\ \ {\it Let $F$ be a Minkowski norm induced by the isoparametric foliation $M_t$
on $(S^{n-1}(1),g^{\mathrm{st}})$ and $g$ its Hessian metric. Then the foliation $N_t=S_F\cap\mathbb{R}_{>0}M_t$ on $(S_F,g)$ is isoparametric.}
\medskip

Theorem B provides more examples of isoparametric foliations. Indeed, when $M_t$ is homogeneous, i.e., it is induced by the isometric cohomogeneity one action of some compact connected Lie group $G$ (see \cite{HL1971,St1996,TT1972} for its classification), the isometric $G$-action on $(S_F,g)$ is
also of cohomogeneity one. So the $G$-orbits $N_t$ provide an isoparametric foliation on $(S_F,g)$. Though this shortcut to Theorem B is not valid for inhomogeneous $M_t$ of OT-FKM type \cite{FKM1981,OT1975}, it provides the most crucial hint, and it inspires us to more generally study the Hessian isometries between Minkowski norms. It is also remarkable that similar correspondence has been found for isoparametric foliations on smooth homotopy spheres (see Theorem 1.1 in \cite{Ge2016}), where topology rather than geometry or Lie theory plays the main role.

Let $F_1$ and $F_2$ be two Minkowski norms on $\mathbb{R}^n$ with $n\geq2$, and $g_1$ and $g_2$ their Hessian metrics
respectively. Then a {\it Hessian
isometry} $\Phi$ from $F_1$ to $F_2$ is a diffeomorphism on $\mathbb{R}^n\backslash\{0\}$ which is an isometry from $g_1$ to
$g_2$. See Section \ref{subsection-4-1} for its basic properties and local version.
A linear isomorphism $\Phi$ on $\mathbb{R}^n$ satisfying $F_1=F_2\circ\Phi$
naturally induces a Hessian isometry when restricted to
$\mathbb{R}^n\backslash\{0\}$. We call it a {\it linear isometry} from $F_1$ to $F_2$.

As we have seen, linear isometry provides us the hint and shortcut to Theorem B. Besides, it also helps us prove a special case of Laugwitz Conjecture \cite{La1965}, which improves Corollary 1.7 in \cite{XM2020} (see Theorem \ref{main-thm-3}).\medskip

\noindent
{\bf Theorem C.}\ \ {\it Let $F$ be a Minkowski norm on $\mathbb{R}^n$ with $n>2$ induced by the isoparametric foliation $M_t$ on
$(S^{n-1}(1),g^{\mathrm{st}})$. Suppose its Hessian metric $g$ is
flat on $\mathbb{R}^n\backslash\{0\}$, then $F$ is Euclidean.}
\medskip

The (possibly) nonlinear Hessian isometry between
two Minkowski norms induced by $M_t$ is more intriguing.
Generally speaking, its complete classification
is a hard problem which involves complicated case-by-case discussion. In this paper, we only concentrate in a subclass, i.e., we consider the triple $(F_1,F_2,\Phi)$, in which
$F_1$ and $F_2$ are Minkowski norms induced by $M_t$, and the Hessian isometry $\Phi$ from $F_1$ to $F_2$ preserves the orientation
and fixes the spherical $\xi$-coordinates. There are two ways to describe this triple.

We may start with the spherical coordinates presentations for $(F_1,F_2,\Phi)$, i.e., $F_1=r\sqrt{2f(t)}$, $F_2=r\sqrt{2h(\theta)}$ (we use $\theta$ to denote the spherical $t$-coordinate for $F_2$), and $\Phi:(r,t,\xi)\mapsto (\tfrac{rf(t)^{1/2}}{h(\theta(t))^{1/2}},\theta(t),\xi)$,
we find that $(f(t),h(\theta),\theta(t))$ must satisfy the $D_{2d}$-symmetry and the following ODE system,
\begin{eqnarray}
& &\tfrac{1}{2f(t)}\tfrac{{\rm d}^2}{{\rm d}t^2}f(t)-\tfrac{1}{4f(t)^2}\left(\tfrac{{\rm d}}{{\rm d}t}f(t)\right)^2+1\nonumber\\
&=&\left(\tfrac{{\rm d}}{{\rm d}t}\theta(t)\right)^2\
\left(\tfrac{1}{2h(\theta(t))}\tfrac{{\rm d}^2}{{\rm d}\theta^2}h(\theta(t))-\tfrac{1}{4h(\theta(t))^2}\left(\tfrac{{\rm d}}{{\rm d}\theta}h(\theta(t))\right)^2+1\right),\quad\mbox{and}
\label{1001}\\
   & &\sin^2 (t+\tfrac{k\pi}{d})+\tfrac{\cos(t+\tfrac{k\pi}{d})
\sin(t+\tfrac{k\pi}{d})}{2f(t)}\tfrac{{\rm d}}{{\rm d}t}f(t)\nonumber\\
&=&\sin^2(\theta(t)+\tfrac{k\pi}{d})+
\tfrac{\cos(\theta(t)+\tfrac{k\pi}{d})
\sin(\theta(t)+\tfrac{k\pi}{d})}{2h(\theta(t))}
\tfrac{{\rm d}}{{\rm d}\theta}h(\theta(t))
\label{1002}
\end{eqnarray}
for each $k\in\{0,\cdots,d-1\}$.

Alternatively, we may restrict $(F_1,F_2,\Phi)$ to any normal plane $\mathbf{V}$. With $\mathbf{V}$ identified with
$\mathbb{R}^2$ (see Section \ref{subsection-2-4}), we get a triple
$(\overline{F}_1,\overline{F}_2,\overline{\Phi})$ with $D_{2d}$-symmetry, where both $\overline{F}_i$ are Minkowski norms on $\mathbb{R}^2$, and $\overline{\Phi}$ is a Hessian isometry between $\overline{F}_i$. In particular, the ODE (\ref{1002}) can be interpreted as a {\it (d)-property}, defined by
the equality $g_x^{F_1}(x'',x)=g_{\overline{x}}^{F_2}(\overline{x}'',\overline{x})$
for any nonzero $x=x'+x''$ and $\Phi(x)=\overline{x}=\overline{x}'+\overline{x}''$ with respect to
a given orthogonal decomposition $\mathbb{R}^n=\mathbf{V}'+\mathbf{V}''$.
See Definition \ref{defining-(d)-property} in Section \ref{subsection-5-3} and its local version in Section \ref{subsection-6-3}.

Summarizing Theorem \ref{main-thm-4} and Theorem \ref{main-thm-5}, we get the following
complete description for Hessian isometries between two Minkowski norms induced by $M_t$, which preserve the orientation and fix the $\xi$-coordinates.\medskip

\noindent
{\bf Theorem D.} \ \ {\it Let $M_t$ be any isoparametric foliation on
$(S^{n-1}(1),g^{\mathrm{st}})$ with $d$ principal curvatures.
Then there are one-to-one correspondences between any two of the following three sets:
\begin{enumerate}
\item The set of all triples $(F_1,F_2,\Phi)$, in which both $F_i$ are
Minkowski norms induced by $M_t$, and $\Phi$ is a Hessian isometry from $F_1$ to $F_2$ which preserves the orientation and
fixes the spherical $\xi$-coordinates;
\item The set of all triples $(f(t),h(\theta),\theta(t))$, such that $f(t)$ and $h(\theta)$ are $D_{2d}$-invariant positive smooth functions
    on $\mathbb{R}/(2\mathbb{Z}\pi)$ satisfying the requirement in
    Theorem A, $\theta(t)$ is a $D_{2d}$-equivariant orientation preserving diffeomorphism on $\mathbb{R}/(2\mathbb{Z}\pi)$ fixing each point in $\tfrac{\mathbb{Z}\pi}d$, and the triple is a solution of the ODE system for all $t\in\mathbb{R}/(2\mathbb{Z}\pi)$, which consists of  (\ref{1001}) and (\ref{1002}) for all $k\in\{0,\cdots,d-1\}$;
\item The set of all triples $(\overline{F}_1,\overline{F}_2,\overline{\Phi})$, in which both $\overline{F}_i$ are $D_{2d}$-invariant Minkowski norms on
    $\mathbb{R}^2$, and $\overline{\Phi}$ is a $D_{2d}$-equivariant orientation preserving Hessian isometry from $\overline{F}_1$ to
    $\overline{F}_2$ which satisfies the (d)-property with respect to the decomposition $$\mathbb{R}^2=\mathbf{V}'+\mathbf{V}''=
\mathbb{R}(\cos(-\tfrac{k\pi}d),\sin(-\tfrac{k\pi}d))
+\mathbb{R}(\cos(\tfrac{\pi}2-\tfrac{k\pi}d),\sin(\tfrac{\pi}2-
\tfrac{k\pi}{d})))$$
for each $k\in\{0,\cdots,d-1\}$.
\end{enumerate}
The correspondences from (1) and (3) to (2) are provided by the spherical and polar coordinates presentations respectively. The correspondence between (1) and (3) is provided by the restriction
to any normal plane $\mathbf{V}$ for $M_t$ and an identification
between $\mathbf{V}$ and $\mathbb{R}^2$.}
\medskip

Finally, we consider the construction for the Hessian isometry $\Phi$
in Theorem D.

Legendre transformation (or its composition with a positive scalar multiplication) provides an important class of (possibly) nonlinear Hessian isometries \cite{Sc2013}. Notice that in this paper we have
used the standard
inner product to identify $\mathbb{R}^n$ with its dual. So for
a Minkowski norm $F$ induced by $M_t$, its dual is also a Minkowski norm on $\mathbb{R}^n$ induced by $M_t$, and its Legendre transformation
preserves the orientation and fixes the spherical $\xi$-coordinates. Theorem D (or Theorem \ref{main-thm-5}) provides the one-to-one correspondence between Legendre transformations for Minkowski norms induced by $M_t$ and Legendre transformations for $D_{2d}$-invariant Minkowski norms on $\mathbb{R}^2$. See Lemma \ref{lemma-10} and Theorem \ref{main-thm-7} for the precise statements.

More examples for the Hessian isometry $\Phi$ in Theorem D  can be constructed by gluing positive scalar multiplications and the compositions of Legendre transformations and
positive scalar multiplications (see Remark \ref{remark-7-3}).

On the other hand, when we have $d>2$ for the foliation $M_t$, this
gluing construction can exhaust all the wanted $\Phi$.
The method for discussing the ODE system consisting of (\ref{1001}) and (\ref{1002}) with $k=0$ (which corresponds to the (d)-property with $k=0$ in (3) of Theorem D) has been given
in \cite{XM2020}, which enables us to locally determine the triple $(f(t),h(\theta),\theta(t))$ around a generic $t_0\in(0,\tfrac{\pi}d)$.
By the assumption $d>2$, (3) in Theorem D requires essentially more (d)-properties for the triple $(\overline{F}_1,\overline{F}_2,\overline{\Phi})$
(for example, the one with $k=1$). Applying Lemma \ref{lemma-12} accordingly, we prove the following theorem verifying our observation (see Theorem \ref{main-thm-6} for the more precise statement).\medskip

\noindent
{\bf Theorem E.}\ \ {\it Any Hessian isometry between two Minkowski norms induced by an isoparametric foliation on
$(S^{n-1}(1),g^{\mathrm{st}})$ with $d>2$, which preserves the orientation and fixes the spherical $\xi$-coordinates can be constructed by gluing positive scalar multiplications and compositions between the Legendre transformation of $F_1$ and positive scalar multiplications. In particular, it satisfies the (d)-property for any orthogonal decomposition of $\mathbb{R}^n$.}
\medskip

When $d=1$ or $d=2$, Theorem 1.4 and Theorem 1.5 in \cite{XM2020} provide a similar local description for $\Phi$.

At the end, we remark that most results in this paper for Hessian isometries are also valid for local Hessian isometries. To avoid iterance and complexity of terminology and notations, we skip those details.

This paper is organized as following. In Section 2, we introduce the spherical coordinates and spherical local frame induced by an isoparametric foliation $M_t$ on the unit sphere. In Section 3, we introduce the Minkowski norm induced by $M_t$, and prove Theorem A and Theorem B. In Section 4, we introduce the notion of Hessian isometry and prove Theorem C. In Section 5, we study Hessian isometries between two Minkowski norms induced by $M_t$ and prove Theorem D for those which preserve  the orientation and fix  the spherical $\xi$-coordinates. In Section 6, we discuss  Legendre transformation and (d)-property. In Section 7, we use the ODE method and (d)-property to
provide the local description for the Hessian isometry $\Phi$ in Theorem D when $d>2$, and prove Theorem E.
\section{Spherical coordinates and spherical local frame induced by
an isoparametric foliation on the unit sphere}

\subsection{Isoparametric function and isoparametric foliation}
An {\it isoparametric function} on a Riemannian manifold $(M,{g})$ is a smooth function
$p:M\rightarrow\mathbb{R}$ such that it is regular almost everywhere, and its gradient vector field $\mathrm{grad}\ p$ and its Laplacian $\Delta p$ satisfy
$$g(\mathrm{grad}\ p,\mathrm{grad}\ p)=a\circ p\quad\mbox{and}\quad
\Delta p=b\circ p$$ for some one-variable functions $a(s)$ and $b(s)$. For each regular value $s$ of $p$, its pre-image $M_s=p^{-1}(s)$ is called an {\it isoparametric hypersurface} \cite{CR2015}.
We will also use $M_s$ to denote the {\it isoparametric foliation} (i.e., the set of all non-empty $M_s$). A geodesic is called {\it normal} for (the foliation) $M_s$, if it intersects each $M_s$ orthogonally.

The isoparametric foliation is called {\it homogeneous} if there exists
a Lie group $G$ of isometric actions on $(M,{g})$ such that each $M_s$
is a $G$-orbit, i.e., this isoparametric foliation is induced by the cohomogeneity-one isometric action of $G$. Indeed, any cohomogeneity-one isometric action can locally induce an isoparametric foliation.
\subsection{Isoparametric foliation on a unit sphere}
On an Euclidean space $\mathbb{R}^n$ with $n\geq 2$, we have the standard Euclidean inner product $\langle\cdot,\cdot\rangle$, the standard Euclidean norm $|\cdot|=\langle\cdot,\cdot\rangle^{1/2}$  and the orthonormal linear coordinates $x=(x_1,\cdots,x_n)$. Meanwhile, we have the standard
flat metric on $\mathbb{R}^n$,
${g^{\mathrm{st}}}={\rm d}x_1^2+\cdots+{\rm d}x_{n}^2$. We will also use $g^{\mathrm{st}}$ to denote its restrictions to submanifolds.

Any isoparametric foliation on $(S^{n-1}(1),{g^{\mathrm{st}}})$
can be related to an
isoparametric function $p: S^{n-1}(1)\rightarrow[-1,1]$, which is the restriction of a homogeneous polynomial of degree $d\in\{1,2,3,4,6\}$ on $\mathbb{R}^n$, where $d$ is the number of principal curvatures. Further more, $\pm1$ are the only critical values of $p(\cdot)$. In this foliation,  each $M_s=p^{-1}(s)$ with $-1<s<1$ is a closed connected isoparametric hypersurface, and the two critical sets
$M_{\pm1}$ are the two {\it focal submanifolds} \cite{Mu1980}.

There are only two subclasses of isoparametric foliations on the unit spheres \cite{Ch2020}. One subclass are those homogeneous ones, which were classified in \cite{HL1971,TT1972}. The other subclass are of the OT-FKM type \cite{FKM1981,OT1975}. Notice that the OT-FKM type must have $d=4$, and there is some overlap between
the subclasses.

%The degree $d$ is an important index characterizing the isoparametric foliation. Each $M_s$ with $-1<s<1$ has exactly $d$ distinct principal
%curvature values \cite{Mu1980}. So we may simply say $M_s$ is an isoparametric foliation {\it with $d$ principal curvatures}.

Consider any maximal extended normal geodesic $\gamma\subset (S^{n-1}(1),g^{\mathrm{st}})$ for  $M_s$. It is a great circle, i.e., the intersection between a plane $\mathbf{V}$ passing the origin and $S^{n-1}(1)$. We will simply call this $\mathbf{V}$ a {\it normal plane} for (the foliation) $M_s$, because it coincides with the orthogonal normal complement of $T_xM_s$ in $\mathbb{R}^n=T_x(\mathbb{R}^n\backslash\{0\})$ for $x\in\gamma\cap M_s$.
The intersection $\gamma\cap(M_{-1}\cup M_1)$ is the set of a pair of antipodal points when $d=1$, or the vertex set of a regular $2d$-polygon when $d>1$. The points in $\gamma\cap M_{-1}$ and in $\gamma\cap M_{1}$ appear alternatively along $\gamma$. Denote
$\mathrm{dist}_{S^{n-1}(1)}(\cdot,\cdot)$ and $\mathrm{dist}_{\gamma}(\cdot,\cdot)$ the distance functions on
$(S^{n-1}(1),{g^{\mathrm{st}}})$ or $(\gamma,{g^{\mathrm{st}}})$ respectively. Then we have
$$\mathrm{dist}_{S^{n-1}(1)}(M_{-1},M_1)=\mathrm{dist}_{\gamma}(\gamma\cap M_{-1}, \gamma\cap M_1)=\tfrac{\pi}{d}.$$
For any $s\in(-1,1)$, we have $c=\mathrm{dist}_{S^{n-1}(1)}(M_{-1},M_s)\in(0,\tfrac{\pi}{d})$, and
$$\gamma\cap M_s=\{x\in\gamma| \mathrm{dist}_{\gamma}(x,\gamma\cap M_{-1})=c\}$$
contains $2d$ points. The principal curvatures of $M_s$, with respect to  the normal direction represented by $\mathrm{grad}\ p$, are exactly $\cot(c+\tfrac{k\pi}{d})$, $k=0,\cdots,d-1$. The multiplicities of these principal curvatures are crucial for the classification theory, which has been extensively studied \cite{Ab1983,Fa2017,St1999,Ta1991}.
\label{subsection-2-2}
\subsection{Parametrization for an isoparametric foliation
on the unit sphere}
In later discussion, we will always parametrize an isoparametric
foliation on $(S^{n-1}(1),g^{\mathrm{st}})$ with $d$ principal curvatures
as $M_t$ with $t\in[0,\tfrac{\pi}d]$, so that for any $x\in M_t$ we have $$\mathrm{dist}_{S^{n-1}(1)}(x,M_0)=
\mathrm{dist}_{S^{n-1}(1)}(M_t,M_0)=t.$$
By this parametrization, $M_0$ and $M_{{\pi}/d}$ are the two focal submanifolds, and all other $M_t$ are isoparametric hypersurfaces.
Restricted to each normal geodesic segement realizing the distance
from $M_0$ to $M_{t}$ with $0<t\leq \tfrac{\pi}d$, $t$ is an $g^{\mathrm{st}}$-arc length parameter.

Notice that the parameter $t$, when presented as
 $t(x)=\mathrm{dist}_{S^{n-1}(1)}(x,M_0)$,
 is an isoparametric function on $(S^{n-1}(1)\backslash(M_0\cup M_{\pi/d}),g^{\mathrm{st}})$ for the foliation $M_t$.

\subsection{Identification between a normal plane and $\mathbb{R}^2$ with $D_{2d}$-action}
\label{subsection-2-4}
Later we will frequently use the following identification between $\mathbb{R}^2$ and a normal plane for the isoparametric foliation $M_t$ on $(S^{n-1}(1),g^{\mathrm{st}})$.

Let $\mathbf{V}\subset\mathbb{R}^n$ be a normal plane for $M_t$, i.e., $\gamma=\mathbf{V}\cap S^{n-1}(1)$ is a normal geodesic for $M_t$ in $(S^{n-1}(1),g^{\mathrm{st}})$. We parametrize $\gamma$ as
$\gamma(t)$ by its $g^{\mathrm{st}}$-arc length, and require $\gamma(0)\in M_0$.
Then $v_1=\gamma(0)$ and $v_2=\gamma(\tfrac{\pi}{2})$, when they are viewed as unit vectors, provide an orthonormal basis
for $\mathbf{V}$. We will identify $\mathbf{V}$ with $\mathbb{R}^2$ such that $v_1$ and $v_2$
are identified with the standard orthonormal basis vectors $e_1=(1,0)$
and $e_2=(0,1)$ respectively.

This identification depends on the choice of $v_1$ from $\gamma\cap M_0$ and the direction of curve
$\gamma(t)$. Changing $v_1$ and changing the direction of $\gamma(t)$ result linear isometries of $\mathbb{R}^2$ which belong to the finite group $D_{2d}$.
On the normal plane $\mathbf{V}$ for $M_t$, $D_{2d}$ is the group of all linear isometries
which preserves $\gamma\cap M_0$. It is the dihedral group when $d>1$, and $\mathbb{Z}_2$ when $d=1$. For $\mathbb{R}^2$,
$D_{2d}$ is generated by the right multiplications by
$$\left(
    \begin{array}{cc}
      1 & 0 \\
      0 & -1 \\
    \end{array}
  \right)\quad\mbox{and}\quad
  \left(
    \begin{array}{cc}
      \cos\tfrac{2\pi}{d} & \sin\tfrac{2\pi}d \\
      -\sin\tfrac{2\pi}d & \cos\tfrac{2\pi}d \\
    \end{array}
  \right)
$$
on row vectors. Alternatively, when we use the polar coordinates
$(r,t)\in\mathbb{R}_{>0}\times \mathbb{R}/(2\mathrm{Z}\pi)$ for $x=(x_1,x_2)=(r\cos t,r\sin t)$ on $\mathbb{R}^2$ or $x=x_1v_1+x_2v_2=r\cos t \ v_1+r\sin t\ v_2$ on $\mathbf{V}$, we have the corresponding $D_{2d}$-action on the space $\mathbb{R}/(2\mathrm{Z}\pi)$ of all polar $t$-coordinates, which
is generated by the mappings $t\mapsto -t$ and $t\mapsto t+\tfrac{2\pi}d$.

Generally speaking, as long as the subjects we discuss later have the $D_{2d}$-symmetry (i.e., $D_{2d}$-invariancy or $D_{2d}$-equivariancy), then it does not depend on the identification when we translate them from $\mathbf{V}$ to $\mathbb{R}^2$ or vice versa.
\subsection{Spherical coordinates and spherical local frame}
\label{subsection-2-5}
Let $M_t$ be an isoparametric foliation on
$(S^{n-1}(1),g^{\mathrm{st}})$ with $d$ principal curvatures.
For any point $x$ in the conic open subset  $$C(S^{n-1}(1)\backslash(M_0\cup M_{\pi/d}))=\mathbb{R}_{>0}  (S^{n-1}(1)\backslash(M_0\cup M_{\pi/d}))=\mathbb{R}^n\backslash(\mathbb{R}_{\geq0}M_0\cup
\mathbb{R}_{\geq0}M_{\pi/d}),$$
its {\it spherical coordinates} induced by (the foliation) $M_t$,  $(r,t,\xi)\in\mathbb{R}_{>0}\times (0,\tfrac{\pi}{d})\times M_{{\pi/2d}}$, is determined by the following requirements:
$r=|x|>0$, $x/|x|\in M_t$, and there exists a normal geodesic segment in $(S^{n-1}(1),g^{\mathrm{st}})$ for the given foliation, which connects $x$ to $\xi$ without passing the two focal submanifolds.
The mapping from $x$ to its spherical
coordinates $(r,t,\xi)$ is a diffeomorphism between $C(S^{n-1}(1)\backslash(M_0\times M_{\pi/d}))$ and
$\mathbb{R}_{>0}\times (0,\tfrac{\pi}d)\times M_{\pi/2d}$.

Then we construct the local frame in $C(S^{n-1}(1)\backslash(M_0\cup M_{\pi/d}))$ with spherical $\xi$-coordinates contained in some sufficiently small open subset
$U$ of $M_{\pi/2d}$.
Here {\it local frame} means a set of smooth tangent vector fields defined on the same open subset,
which values at each point provide a basis of the tangent space.

For the spherical $r$- and
$t$-coordinates induced by $M_t$, we have the tangent vector fields $\partial_r$ and $\partial_t$. That means, $\partial_r$ generates the rays initiating from the origin, and $\partial_t$ generates
the normal geodesics for the foliation $M_{r,t}=r M_t=\{rx| \forall x\in M_t\}$ on $(S^{n-1}(r)\backslash(M_{r,0}\cup  M_{r,\pi/d}),{g^{\mathrm{st}}})$.
Obviously we have $|\partial_r|^2={g^{\mathrm{st}}}(\partial_r,\partial_r)=1$ and
$[\partial_r,\partial_t]=0$. Our convention for parametrizing $M_t$ implies $|\partial_t|^2=r^2$.

The other local tangent vector fields, $X_1,\cdots,X_{n-2}$, are tangent to the foliation $M_{r,t}$. Firstly, we construct
them on $U\subset M_{{\pi/2d}}$, such that the following are satisfied:
\begin{enumerate}
\item Each $X_i$ is a tangent vector field of constant length on $U$;
\item For each $i$, there is a principal curvature value \begin{equation}\label{principoal-curvature-constant-kappa-i}
    \kappa_i=\cot(t+\tfrac{k_i\pi}{d}),\quad k_i\in\{0,\cdots,d-1\},
    \end{equation}
    for $M_t$ in $(S^{n+1}(1),{g^{\mathrm{st}}})$, with respect to  the normal direction $\partial_t$, such that the value of $X_i$ at each point of $U$ is a
 eigenvector for the eigenvalue $\kappa_i$ of the shape operator;
\item At each point of $U$, the values of $X_i$ for all $1\leq i\leq n-2$ provide a $g^{\mathrm{st}}$-orthogonal basis for the tangent space of $M_{{\pi/2d}}$.
\end{enumerate}
Notice that any $\xi\in M_{\pi/2d}$ has a sufficient small neighborhood $U$ in $M_{\pi/2d}$ with no topological
obstacle to the above construction of $X_1,\cdots,X_{n-2}$.
If we ignore the multiplicities, then we have
$\{k_1,\cdots,k_{n-2}\}=\{0,\cdots,d-1\}$.

Then we extend each $X_i$ such that
$$[\partial_r,X_i]=[\partial_t,X_i]=0,\quad\forall 1\leq i\leq n-2.$$
Indeed, if $X_i$ generates the local diffeomorphisms $\rho_s$ on $M_{{\pi/2d}}$,  then after the extension, it generates $\tilde{\rho}_s$ with the spherical coordinates presentation
$\tilde{\rho}_s(r,t,\xi)=(r,t,\rho_s(\xi))$. At each point, $X_1,\cdots,X_{n-2}$ linearly span the tangent space of $M_{r,t}$. So their brackets $[X_i,X_j]$ are tangent to
the foliation $M_{r,t}$ as well. For simplicity, we denote
$$[X_i,X_j]\equiv 0\ \mbox{(mod}\ X_1,\cdots,X_{n-2}\mbox{)}.$$

Along any normal geodesic for $M_t$ in $(S^{n-1}(1),{g^{\mathrm{st}}})$,
each $X_i$ can be extended to a Jacobi field on the whole great circle,
which vanishes at one of $M_0$ and $M_{\pi/d}$. So,  $X_i/|X_i|$ is parallel along each $t$-curve (i.e. the spherical $r$- and $\xi$-coordinates are fixed) for ${g^{\mathrm{st}}}$ on $\mathbb{R}^n\backslash\{0\}$. By the positive 1-homogeneity,  we may denote
\begin{equation}\label{formula-X_i^2}
|X_i|^2={g^{\mathrm{st}}}(X_i,X_i)=r^2 f_i(t)\quad\mbox{with}\quad
f_i(t)=a_i \sin^2(t+\tfrac{k_i\pi}{d}),
\end{equation}
where $a_i$ is some positive constant and $k_i\in\{0,\cdots,d-1\}$ is the integer in (\ref{principoal-curvature-constant-kappa-i}).

Now
we have constructed the ${g^{\mathrm{st}}}$-orthogonal local
frame $\{\partial_r,\partial_t,X_1,\cdots,X_{n-2}\}$.
For simplicity, we will call it a
{\it spherical local frame} induced by $M_t$.
Above discussion can be summarized as the following lemma.

\begin{lemma} \label{lemma-1}
Let $\{\partial_r,\partial_t,X_1,\cdots,X_{n-2}\}$
be a spherical local frame induced by the isoparametric foliation $M_t$ on $(S^{n-1}(1),g^{\mathrm{st}})$ with $d$ principal curvatures. Then we have the following:
\begin{enumerate}
\item The brackets among tangent vector fields in this spherical local frame satisfy
    \begin{eqnarray*}
    & &[\partial_r,\partial_t]=0, \quad [\partial_r,X_i]=[\partial_t,X_i]=0,\, \forall i,\\
    & &[X_i,X_j]\equiv 0\, \mathrm{(mod}\  X_1,\cdots,X_{n-2}\mathrm{)},\, \forall i,j.
    \end{eqnarray*}
\item The standard flat metric
${g^{\mathrm{st}}}$ on $\mathbb{R}^{n}\backslash\{0\}$ can be presented as
$${g^{\mathrm{st}}}={d}r^2+r^2{d}t^2
+r^2f_1(t)\theta_1^2+\cdots+r^2 f_{n-2}(t)\theta_{n-2}^2,$$
where $\{{\rm d}r,{\rm d}t,\theta_1,\cdots,\theta_{n-2}\}$
is the dual frame for $\{\partial_r,\partial_t,X_1,\cdots,X_{n-2}\}$,
and $f_i(t)$ is given in (\ref{formula-X_i^2}).
\end{enumerate}
\end{lemma}

Using Lemma \ref{lemma-1}, we can further calculate
the Levi-Civita connection of $(\mathbb{R}^n\backslash\{0\},{g^{\mathrm{st}}})$.

\begin{lemma}\label{lemma-2}
For the Levi-Civita connection ${\tilde{\nabla}}$  of $(\mathbb{R}^n\backslash\{0\},{g^{\mathrm{st}}})$, we have
\begin{eqnarray*}
& &{\tilde{\nabla}}_{\partial_r}\partial_r=0,\quad
{\tilde{\nabla}}_{\partial_r}\partial_t=
{\tilde{\nabla}}_{\partial_t}\partial_r=\tfrac1r\partial_t,\quad
{\tilde{\nabla}}_{\partial_t}\partial_t=-r\partial_r,\\
& &{\tilde{\nabla}}_{\partial_r}X_i={\tilde{\nabla}}_{X_i}\partial_r=
\tfrac1r X_i,\ \forall i,\quad
{\tilde{\nabla}}_{\partial_t}X_i={\tilde{\nabla}}_{X_i}\partial_t
=\tfrac{1}{2f_i(t)}\tfrac{{\rm d}}{{\rm d}t}f_i(t)  X_i,\ \forall i,\\
& &{\tilde{\nabla}}_{X_i}X_i\equiv-rf_i(t)\partial_r
-\tfrac{1}{2}\tfrac{{\rm d}}{{\rm d}t}f_i(t)\partial_t\ \mathrm{(mod}\ X_1,\cdots,X_{n-2}\mathrm{)},\ \forall i,\\
& &{\tilde{\nabla}}_{X_i}X_j\equiv0\
\mathrm{(mod}\ X_1,\cdots,X_{n-2}\mathrm{)},\ \forall i\neq j.
\end{eqnarray*}
\end{lemma}
\section{Minkowski norm induced by an isoparametric foliation}

\subsection{Minkowski norm and Hessian metric}
\label{subsection-3-1}
A {\it Minkowski norm} $F$ on $\mathbb{R}^{n}$ with $n\geq 2$
is a continuous
function satisfying the following conditions:
\begin{enumerate}
\item Positiveness and smoothness. The restriction of $F$ to $\mathbb{R}^{n}\backslash\{0\}$ is positive and smooth.
\item Positive 1-homogeneity. For any $\lambda\geq0$ and any $x\in\mathbb{R}^{n }$, $F(\lambda x)=\lambda F(x)$.
\item Strong convexity. For the linear coordinates $x=(x_1,\cdots,x_{n})$, the Hessian matrix
$\left(\tfrac{{\partial}^2 }{{\partial}x_i{\partial}x_j}E\right)$ for $E=\tfrac12F^2$, which is also called the {\it fundamental tensor}
\cite{Be1941}, is positive definite at any $x\neq0$.
\end{enumerate}
By its positive 1-homogeneity and strong convexity,
the Minkowski norm $F$ is totally determined by its
{\it indicatrix},  $S_F=\{x\in\mathbb{R}^{n }| F(x)=1\}$, which is a smooth convex sphere surrounding the origin.
On $\mathbb{R}^{n}\backslash\{0\}$, the Minkowski norm $F$
determines a Riemannian metric $g =g(\cdot,\cdot)$, such that
for any $u,v\in \mathbb{R}^{n}=T_x(\mathbb{R}^{n}\backslash\{0\})$, we have
\begin{equation*}
g(u,v)=\tfrac{\partial^2}{\partial s\partial t}{}|_{s=t=0} \left(\tfrac12 F(x+su+tv)^2\right)
\end{equation*}
at $x\in\mathbb{R}^{n}\backslash\{0\}$. We call $g$
the {\it Hessian metric} of $F$ and use the same $g$ to denote its restriction to submanifolds of $\mathbb{R}^n\backslash\{0\}$. Sometimes, when the norm $F$ and the nonzero base vector $x$ need to be specified, we denote it as $g^F_x(\cdot,\cdot)$.

For example,
a Minkowski norm $F$ is {\it Euclidean} if and only if
its Hessian matrices are irrelevant to the choice of $x\in\mathbb{R}^{n}\backslash\{0\}$, i.e., the Hessian metric $g$ can be viewed as an inner product on $\mathbb{R}^n$ which satisfies $F(x)\equiv g(x,x)^{1/2}$ for every $x\in\mathbb{R}^n$.
An {\it $(\alpha,\beta)$-norm}
is of the form $F=\alpha\varphi(\tfrac{\beta}{\alpha})$, in which $\alpha$ is an Euclidean norm, $\beta$ is a homogeneous linear function, and $\varphi(s)$ is some positive one-variable function. An $(\alpha_1,\alpha_2)$-norm
is of the form
$F(x)=\alpha(x)\varphi(\tfrac{\alpha(x_1)}{\alpha(x)})$, where $\alpha$ is an Euclidean norm, $\varphi(s)$ is some one-variable function, and $x=x_1+x_2$ is with respect to a fixed
$\alpha$-orthogonal decomposition $\mathbb{R}^{n}=\mathbf{V}_1+\mathbf{V}_2$.

\subsection{Induced Minkowski norm and a criterion theorem}

Let $M_t$ be an isoparametric foliation
on $(S^{n-1}(1),g^{\mathrm{st}})$ with $d$ principal curvatures. Then
we can construct a Minkowski norm $F$ on $\mathbb{R}^n$, requiring it to be constant on each $M_t$. For simplicity, we call it a {\it Minkowski norm induced by  $M_t$}. Using the spherical coordinates, we can present it as $F=r\sqrt{2 f(t)}$, where $f(t)$ is some positive function on $[0,\tfrac{\pi}{d}]$.

Let us more closely observe the two features in the norm $F=r\sqrt{2 f(t)}$, the foliation $M_t$ and the function $f(t)$.

Firstly, when $M_t$ satisfies $d=1$ or $2$. The induced Minkowski norm is an $(\alpha,\beta)$- or an $(\alpha_1,\alpha_2)$-norm. Indeed,
by choosing suitable inner product, all $(\alpha,\beta)$- and $(\alpha_1,\alpha_2)$-norms can be induced by an isoparametric foliation on the unit sphere with $d=1$  or $d=2$.

Secondly, the function $f(t)$  for  $F=r\sqrt{2f(t)}$ can be determined by the restriction of $F$ to any maximally extended normal geodesic for $M_t$ in $(S^{n-1}(1),{g^{\mathrm{st}}})$. So $f(t)$ can be extended to
a $D_{2d}$-invariant positive smooth function for $t\in \mathbb{R}/(2\pi\mathbb{Z})$, i.e., we always have $f(t)=f(-t)$ and
$f(t)=f(t+\tfrac{2\pi}d)$.
The function $f(t)$ after extension is used in the polar coordinates
presentation $\overline{F}=r\sqrt{2f(t)}$ for the restriction
$\overline{F}=F|_{\mathbf{V}}$
of $F$ to any normal plane $\mathbf{V}$ for $M_t$.

The following criterion theorem answers exactly and explicitly when the formal expression $F=r\sqrt{2f(t)}$ really defines a Minkowski norm.

\begin{theorem} \label{main-thm-1}
Let $M_t$ be an isoparametric foliation on $(S^{n-1}(1),{g^{\mathrm{st}}})$ with $d$ principal curvatures. Then the following are equivalent:
\begin{enumerate}
\item The function $f(t)$ on $[0,\tfrac{\pi}{d}]$ defines a Minkowski norm $F=r\sqrt{2f(t)}$ on $\mathbb{R}^{n}$ induced by  $M_t$;
\item The function $f(t)$ can be extended to a $D_{2d}$-invariant positive smooth function on $\mathbb{R}/(2\mathbb{Z}\pi)$ such that
\begin{equation}\label{eq:convexity-requirement}
2f(t)\tfrac{{\rm d}^2}{{\rm d}t^2}f(t)-\left(\tfrac{{\rm d}}{{\rm d}t}f(t)\right)^2+4f(t)^2>0
\end{equation}
is satisfied everywhere;
\item The function $f(t)$ after extension (as indicated in (2)) defines a $D_{2d}$-invariant Minkowski norm on $\mathbb{R}^2$ with the polar coordinates presentation $\overline{F}=r\sqrt{2f(t)}$.
\end{enumerate}
\end{theorem}

Its proof is postponed to Section \ref{subsection-3-4}.

The special cases of Theorem \ref{main-thm-1} with $d=1$ or $2$ reprove the following known results for $(\alpha,\beta)$-norms (see the discussion for Proposition 5 in \cite{Cr2020}) and $(\alpha_1,\alpha_2)$-norms (see Theorem 3.2 in \cite{DX2016}).

\begin{corollary}\label{cor-1}
Let $\alpha$, $\beta$ and $\varphi(s)$ be an Euclidean norm,
a nonzero homogeneous linear function on $\mathbb{R}^n$ with $n\geq2$, and a positive one-variable function respectively.
Then $F=\alpha\varphi(\tfrac{\beta}{\alpha})$
defines a Minkowski norm if and only if $\varphi(s)$ is a positive smooth function on $[-b,b]$, where
$b$ is the $\alpha$-norm of $\beta$,
 and $\varphi(s)$
satisfies
\begin{equation}\label{0036}
\varphi(s)-s\tfrac{{\rm d}}{{\rm d}s}\varphi(s)+
(b^2-|s|^2)\tfrac{{\rm d}^2}{{\rm d}s^2}\varphi(s)>0
\end{equation}
on $[-b,b]$.
\end{corollary}

\begin{proof}We first prove Corollary \ref{cor-1} when $n>2$.

Using the
suitable $\alpha$-orthonormal coordinates $(x_1,\cdots,x_n)$ and the corresponding spherical $r$- and $t$-coordinates, the expression
$F=\alpha\varphi(\tfrac{\beta}{\alpha})$ can be changed to $F=r\varphi(b\cos t)=r\sqrt{2f(t)}$ with $f(t)=\tfrac{1}{2}\varphi(b\cos t)^2$ and  $t\in[0,\pi]$. Direct calculation shows
\begin{equation}\label{0035}
2f(t)\tfrac{{\rm d}^2}{{\rm d}t^2}f(t)-\left(\tfrac{{\rm d}}{{\rm d}t}f(t)\right)^2+4f(t)^2=\varphi(s)^3\  \left(\varphi(s)-s\tfrac{{\rm d}}{{\rm d}s}\varphi(s)+
(b^2-|s|^2)\tfrac{{\rm d}^2}{{\rm d}s^2}\varphi(s)\right),
\end{equation}
where $s=b\cos t$.

Assume $F=\alpha\varphi(\tfrac{\beta}{\alpha})$ defines a Minkowski norm, then by the equivalence between (1) and (2) in Theorem \ref{main-thm-1} for $d=1$, $f(t)$ is a smooth even function around $t=0$. By L'Hospital Rule and the theory for implicit function, we can find a function $\psi(s)$ which is smooth at $s=0$,
such that $f(t)=\psi(b^2\sin^2 t)$ around $t=0$. Then $\varphi(s)=\psi(b^2-s^2)$ is smooth at $s=b$. The smoothness of $\varphi(s)$ at $s=-b$ can be similarly verified. Checking the other
claims for $\varphi(s)$ in Corollary \ref{cor-1}
are easy routines.

Assume $\varphi(s)$ satisfies the requirements in Corollary \ref{cor-1},
then obviously $f(t)=\tfrac12\varphi(b\cos t)^2$ is a positive smooth
function on $\mathbb{R}$ which satisfies $f(t)=f(-t)$ and $f(t)=f(t+2\pi)$, i.e. $f(t)$ is a function on $\mathbb{R}/(2\mathbb{Z}\pi)$ which is invariant with respect to  the action
of $D_2=\mathbb{Z}_2$.
The inequality (\ref{eq:convexity-requirement}) follows immediately
after (\ref{0035}) and (\ref{0036}). Finally, the equivalent between (1)
and (2) in Theorem \ref{main-thm-1} for $d=1$ tells us $F=\alpha\varphi(\tfrac{\beta}{\alpha})$ is a Minkowski norm.

To summarize, above argument proves Corollary \ref{cor-1} when $n>2$. When $n=2$, we can use the equivalence between (2) and (3)
in Theorem \ref{main-thm-1} and similar argument to prove this corollary.
\end{proof}

\begin{corollary}\label{cor-2}
Let $\alpha$ be an Euclidean norm on $\mathbb{R}^n$ with $n\geq2$,  $\mathbb{R}^n=\mathbf{V}'+\mathbf{V}''$ an $\alpha$-orthogonal decomposition with $0<\dim \mathbf{V}'=m<n$, and $\varphi(s)$ some positive function on $[0,1]$. Then $F(x)=\alpha\varphi(\tfrac{\alpha(x_1)}{\alpha(x)})$ defines a Minkowski norm if and only if both $\varphi(s)$ and $\psi(s)=\varphi(\sqrt{1-s^2})$ are smooth functions on $[0,1]$, and
the inequality $$\varphi(s)-s\tfrac{{\rm d}}{{\rm d}s}\varphi(s)+
(1-|s|^2)\tfrac{{\rm d}^2}{{\rm d}s^2}\varphi(s)>0$$
is satisfied everywhere.
\end{corollary}

In the proof of Corollary \ref{cor-2}, we need to apply
Theorem \ref{main-thm-1} for $d=2$ to discuss the case $2\leq m\leq n-2$. As the argument for each case is very similar to that in the proof of Corollary \ref{cor-1}, we skip the details.

\subsection{Some calculation for the Hessian metric}
The calculation for the Hessian metric $g$ of the Minkowski norm $F=r\sqrt{2f(t)}$ by a spherical local frame induced by $M_t$ is the foundation for later discussion.
It is a useful observation that $g$ is in fact the second covariant derivative of $E=\tfrac12F^2=r^2f(t)$ with respect to
the Levi-Civita connection ${\tilde{\nabla}}$ on $(\mathbb{R}^n\backslash\{0\},{g^{\mathrm{st}}})$, so we have
$$g(X,Y)=X\cdot(Y\cdot E)-({\tilde{\nabla}}_XY)\cdot E,$$
in which $\cdot$ denotes the directional derivative action of vector fields on differentiable functions and $\tilde{\nabla}$ is the Levi-Civita connection on $(\mathbb{R}^n\backslash\{0\},g^{\mathrm{st}})$.
Use Lemma \ref{lemma-2} and notice $X_i\cdot E=0$, we get
\begin{lemma}\label{lemma-3}
Let $\{\partial_r,\partial_t,X_1,\cdots,X_{n-2}\}$ be
a spherical local frame, and $F=r\sqrt{2f(t)}$
a Minkowski norm, induced by the same isoparametric foliation $M_t$ on $(S^{n-1}(1),{g^{\mathrm{st}}})$, then
we have
\begin{eqnarray*}
& &g(\partial_r,\partial_r)=2f(t),\quad
g(\partial_t,\partial_t)=r^2\tfrac{{\rm d}^2}{{\rm d}t^2}f(t)+2r^2f(t),\\
& &g(X_i,X_i)=r^2\left(2f_i(t)f(t)+\tfrac12\tfrac{{\rm d}}{{\rm d}t}f_i(t)
\tfrac{{\rm d}}{{\rm d}t}f(t)\right),\ \forall i,\\
& & g(\partial_r,\partial_t)=r\tfrac{{\rm d}}{{\rm d}t}f(t),
\quad
g(\partial_r,X_i)=g(\partial_t,X_i)=0, \ \forall i,
\quad g(X_i,X_j)=0,\ \forall i\neq j.
\end{eqnarray*}
\end{lemma}

We see from Lemma \ref{lemma-3} that, though the spherical local frame $\{\partial_r,\partial_t,X_1,\cdots,X_{n-2}\}$ may not be
$g$-orthogonal, it is close, i.e., replacing $\partial_t$ with
$T=\partial_t-\tfrac{r}{2f(t)}\tfrac{{\rm d}}{{\rm d}t}f(t)\partial_r$,
then the frame $\{\partial_r,T,X_1,\cdots,X_{n-2}\}$ is $g$-orthogonal. Using Lemma \ref{lemma-3}, the $g$-norm square of $T$ can be easily calculated. To summarize, we have

\begin{lemma}\label{lemma-7}
The local frame $\{\partial_r,T,X_1,\cdots,X_{n-2}\}$
is $g$-orthogonal, in which $T=\partial_t-\tfrac{r}{2f(t)}\tfrac{{\rm d}}{{\rm d}t}f(t)\partial_r$ with
\begin{equation}\label{0007}
g(T,T)=\tfrac{r^2}{2f(t)}\left(2f(t)\tfrac{{\rm d}^2}{{\rm d}t^2}f(t)-\left(\tfrac{{\rm d}}{{\rm d}t}f(t)\right)^2+4f(t)^2\right).
\end{equation}
\end{lemma}

Let $\mathbf{V}$ be any normal plane for $M_t$, which has a nonempty intersection with the defining domain for $\{\partial_r,\partial_t,X_1,\cdots,X_{n-2}\}$.
Then $\partial_t$
can be smoothly extended to the one on $\mathbf{V}\backslash\{0\}$
corresponding to the polar $t$-coordinate.
Since $\partial_r$ can be globally defined on $\mathbb{R}^n\backslash\{0\}$ (corresponding to the spherical $r$-coordinate), we see $T=\partial_t-\tfrac{r}{2f(t)}\tfrac{{\rm d}}{{\rm d}t}f(t)\partial_r$ can be extended to a smooth tangent vector field on $\mathbf{V}\backslash\{0\}$ which is nonvanishing everywhere.
By the positive 1-homogeneity,
the smooth extensions to $\mathbf{V}\backslash\{0\}$ for $X_1,\cdots,X_{n-2}$ can also be observed. It is easy but useful to see that the nonzero values of $X_1,\cdots,X_{n-2}$ provide a basis for $T_{x}M_{r,t}$ when $x\in\mathbf{V}\cap M_{r,t}=\mathbf{V}\cap rM_t$ with every $r>0$ and $t\in[0,\tfrac{\pi}d]$. We summarize these observations
as the following lemma.

\begin{lemma}\label{lemma-14}
Let $\mathbf{V}$ be any normal plane for $M_t$, which has a nonempty intersection with the defining domain for $\{\partial_r,\partial_t,X_1,\cdots,X_{n-2}\}$. Then the restriction of $\{\partial_r,\partial_t,X_1,\cdots,X_{n-2}\}$ to this intersection
can be canonically extended to $\mathbf{V}\backslash\{0\}$ satisfying
the following:
\begin{enumerate}
\item $\partial_t$ and $T=\partial_t-\tfrac{r}{2f(t)}\tfrac{{\rm d}}{{\rm d}t}f(t)\partial_r$ are nonvanishing everywhere on
    $\mathbf{V}\backslash\{0\}$.
\item The nonzero values of $X_1,\cdots,X_{n-2}$ provide the basis
of the tangent space for $M_{r,t}$ for every $r>0$ and $t\in[0,\tfrac{\pi}d]$.
\end{enumerate}
\end{lemma}

\subsection{Proof of Theorem \ref{main-thm-1}}
\label{subsection-3-4}
We will mainly prove the equivalence between (1) and (2) in Theorem
\ref{main-thm-1}. The equivalence between (2) and (3) can be easily observed in the midway.

Firstly, we prove the claim in Theorem \ref{main-thm-1} from (2) to (1). We assume that $f(t)$ has been extended to a $D_{2d}$-invariant positive smooth function on $\mathbb{R}/(2\mathbb{Z}\pi)$ as claimed in Theorem \ref{main-thm-1}, i.e., we have
\begin{eqnarray}
& &f(t)=f(-t),\quad f(t)=f(\tfrac{2\pi}{d}-t),\quad f(t)=f(t+\tfrac{2\pi}{d}),
\label{0000}\\
& &2f(t)\tfrac{{\rm d}^2}{{\rm d}t^2}f(t)-\left(\tfrac{{\rm d}}{{\rm d}t}f(t)\right)^2+4f(t)^2>0\label{0001},
 \end{eqnarray}
for every $t\in\mathbb{R}/(2\mathbb{Z}\pi)$. Then we
prove $F=r\sqrt{2f(t)}$ is a Minkowski norm induced by $M_t$.
Its positiveness, positive 1-homogeneity, and smoothness on $C(S^{n-1}(1)\backslash(M_0\cup M_{\pi/d}))$ are obvious.

To prove its smoothness at $\mathbb{R}_{>0}M_0$, we can argue as following. By the first equality in (\ref{0000})
and an exercise of L'Hospital Rule, there exists
a positive function $\psi(s)$ such that $f(t)=\psi(t^2)$ and $\psi(s)$ is smooth at $s=0$.
As a function on $(S^{n-1}(1)\backslash (M_0\cup M_{\pi/d}),g^{\mathrm{st}})$,
the spherical $t$-coordinate
coincides with the distance $\mathrm{dist}_{S^{n-1}(1)}(\cdot,M_0)$.
Using the exponential map
for the normal bundle of $M_0$ in $(S^{n-1}(1)\backslash M_{\pi/d}, {g^{\mathrm{st}}})$, we see $t^2=\left(\mathrm{dist}_{S^{n-1}(1)}(\cdot,M_0)\right)^2$ can be smoothly extended to a neighborhood of $M_0$ in $S^{n-1}(1)$. So $F=r\sqrt{2f(t)}=r\sqrt{2\psi(t^2)}$ is
smooth at $\mathbb{R}_{>0}M_0$.

Using the second equality in (\ref{0000}) and similar argument, we can prove $F=r\sqrt{2f(t)}$
is smooth at $\mathbb{R}_{>0}M_{\pi/d}$. Then the smoothness
is verified. Since we have already observed the
positiveness, smoothness and positive 1-homogeneity for $F$, we see its indicatrix
$S_F$ is a smooth sphere surrounding the origin.

Now we verify the strong convexity, which is the most essential part of the proof. Let $\mathbf{V}$ be a normal plane for $M_t$. We consider any spherical local frame $\{\partial_r,\partial_t,X_1,\cdots,X_{n-2}\}$ induced by $M_t$, which defining domain has a nonempty intersection with $\mathbf{V}$. Denote $\overline{F}=F|_{\mathbf{V}}$, then
we have the polar coordinates presentation $\overline{F}=r\sqrt{2f(t)}$.
%Denote $c(t)$ the intersection curve between $\mathbf{V}$ and $S_F$ with $c(t)/|c(t)|=\gamma(t)$. We may parametrize $\gamma(t)$ by its arc length with $c(0)\in M_{0}$. So we also have $c(\tfrac{\pi}{d})\in M_{\pi/d})$, and when restricted to $(0,\tfrac{\pi}{d})$, the parameter $t$ coincides with the $t$-coefficient. In later discussion, we sometimes view $c(t)$ as a directional vector.

Using $\{\partial_r,\partial_t,X_1,\cdots,X_{n-2}\}$ and
$T=\partial_t-\tfrac{r}{2f(t)}\tfrac{{\rm d}}{{\rm d}t}f(t)\partial_r$,
the Hessian $g(\cdot,\cdot)$ of $E=\tfrac12F^2$
has the same presentations as in
Lemma \ref{lemma-3} and Lemma \ref{lemma-7}. For simplicity, we denote $g(X,Y)$ for $X,Y\in\{\partial_r,T,X_1,\cdots,X_{n-2}\}$ as $g_{\alpha\beta}$, where $\alpha$ and $\beta$ are the indices in the ordered set $\{r,T,1,\cdots,n-2\}$. For example
$g_{TT}=g(T,T)$,
$g_{ri}=g(\partial_r,X_i)$, etc..
Then $$(g_{\alpha\beta})=
\mathrm{diag}(g_{rr},g_{TT},g_{11},\cdots,g_{n-2,n-2})$$ is a diagonal matrix. By Lemma \ref{lemma-14}, when we restrict our discussion for $(g_{\alpha\beta})$ to $\mathbf{V}$, the notion of $g_{\alpha\beta}$ can be smoothly extended everywhere on
$\mathbf{V}\backslash\{0\}$.

From Lemma \ref{lemma-3}, we know $g_{rr}=2f(t)>0$.
By Lemma \ref{lemma-7}, the left side of the inequality (\ref{0001}) coincides with $\tfrac{2f(t)}{r^2}\  g_{TT}$. So (\ref{0001})
implies the left up $2\times 2$-block $\mathrm{diag}(g_{rr},g_{TT})$ in the Hessian matrix $(g_{\alpha\beta})$ is positive definite.

Before we go on with the discuss the other diagonal entries $g_{ii}$ in $(g_{\alpha\beta})$, we digress to prove
$\overline{F}=F|_{\mathbf{V}}$
is a Minkowski norm on $\mathbf{V}$. Its strong convexity is the only nontrivial issue
for us to consider. Since the polar coordinates presentation $\overline{F}=r\sqrt{2f(t)}$ looks the same as the spherical coordinates presentation for $F$, the Hessian matrix of $\overline{F}$ coincides with the left up $2\times 2$-block in $(g_{\alpha\beta})$.
We have seen its positive definiteness in the previous discussion, so $\overline{F}$ is strongly convex.
Indeed, this simple observation proves the equivalence between (2) and
(3) in Theorem \ref{main-thm-1}.

For the convenience when discussing $g_{ii}$, we denote $N_t=S_F\cap \mathbb{R}_{>0}M_t$ the foliation on $S_F$ induced
by $M_t$. We
parametrize
$S_{\overline{F}}=\mathbf{V}\cap S_{F}$ as $c(t)$ with
$t\in\mathbb{R}/(2\mathbb{Z}\pi)$, equivariantly with respect to  the action
of $D_{2d}$, such that $c(t)\in \mathbf{V}\cap N_t$ for all $t\in[0,\tfrac{\pi}d]$.
Then $\gamma(t)=\tfrac{c(t)}{|c(t)|}$ is a maximal normal geodesic
for $M_t$ on $(S^{n-1}(1),g^{\mathrm{st}})$ parametrized by its
$g^{\mathrm{st}}$-arc length with $\gamma(0)\in M_0$.

By Lemma \ref{lemma-3}, we have for each $1\leq i\leq n-2$,
\begin{eqnarray}
g_{ii}&=&r^2\left( 2f_i(t)f(t)+\tfrac12 \tfrac{{\rm d}}{{\rm d}t}f_i(t)\tfrac{{\rm d}}{{\rm d}t}f(t)\right)\nonumber\\
&=&\tfrac{a_i}{(2f(t))^{1/2}}\
\sin(t+\tfrac{k_i\pi}{d})\
\left({\sin(t+\tfrac{k_i\pi}d)}(2f(t))^{1/2}+
\tfrac{\cos(t+\tfrac{k_i\pi}d)}{(2f(t))^{1/2}}
\tfrac{{\rm d}}{{\rm d}t}f(t)
\right),
\label{0037}
\end{eqnarray}
at $x=c(t)$ with $t\in[0,\tfrac{\pi}d]$.
Here the function $f_i(t)=a_i\sin^2(t+\tfrac{k_i\pi}d)$ with positive constant $a_i$ and $k_i\in\{0,\cdots,n-1\}$ is
given in (\ref{formula-X_i^2}).

The factor $\sin(t+\tfrac{k_i\pi}{d})$ in the right side of (\ref{0037}) is non-negative for $t\in[0,\tfrac{\pi}d]$ and $k_i\in\{0,\cdots,d-1\}$, and it vanishes if and only if
\begin{equation}\label{0002}
\mbox{either}\quad
(t,k_i)=(0,0)\quad\mbox{or}\quad (t,k_i)=(\tfrac{\pi}{d},d-1).
\end{equation}
The factor
${\sin(t+\tfrac{k_i\pi}d)}(2f(t))^{1/2}+
\tfrac{\cos(t+\tfrac{k_i\pi}d)}{(2f(t))^{1/2}}
\tfrac{{\rm d}}{{\rm d}t}f(t)$ coincides with $g^{\overline{F}}_{c(t)}(c(t),\gamma(\tfrac{\pi}2-\tfrac{k_i\pi}{d}))$, i.e., the derivative of $\overline{E}=\tfrac12\overline{F}^2=r^2f(t)$ in the direction of $\gamma(\tfrac{\pi}2-\tfrac{k_i\pi}{d})$ at $x=c(t)$,
with $k_i\in\{0,\cdots,d-1\}$. %We claim $g^{\overline{F}}_{c(t)}(c(t),
%\gamma(\tfrac{\pi}2-\tfrac{k_i\pi}{d}))\geq$
%By the first two equalities of (\ref{0000}) from the $D_{2d}$-invariancy, $S_{\overline{F}}=S_F\cap\mathbf{V}$ is tangent to the direction of $\gamma(\pm\tfrac{\pi}2)$ at $c(0)$, and it is tangent to the direction of  $\gamma(\tfrac{\pi}{d}\pm\tfrac{\pi}{2})$ at $c(\tfrac{\pi}{d})$, i.e., we have
%$$g^{\overline{F}}_{c(0)}(c(0),\gamma(s))\geq 0,\ \forall s\in[-\tfrac{\pi}2,\tfrac{\pi}2],\quad\mbox{and}\quad
%g^{\overline{F}}_{c({\pi}/d)}(c(\tfrac{\pi}d),
%\gamma(s))\geq 0, \forall s\mbox{ and }\tfrac{\pi}d\pm\tfrac{\pi}2,$$
%where the equalities happen at end points for both.

We denote $s(t)\in[\tfrac{\pi}2,\tfrac{3\pi}2]$ for $t\in[0,\tfrac{\pi}d]$ such that $g^{\overline{F}}_{c(t)}(c(t),\gamma(s(t)- {\pi}))=
g^{\overline{F}}_{c(t)}(c(t),\gamma(s(t)))=0$ and $g^{\overline{F}}_{c(t)}(c(t),\gamma(s))>0$ for all $s\in (s(t)-\pi,s(t))$.
By the first two equalities of (\ref{0000}) from the $D_{2d}$-invariancy, we see $s(0)=\tfrac{\pi}2$ and $s(\tfrac{\pi}d)=
\tfrac{\pi}d+\tfrac{\pi}2$. When $t=0$, the interval $(\tfrac{\pi}2-\tfrac{(d-1)\pi}d,\tfrac{\pi}{2})$ is contained in
$(s(0)-\pi,s(0))$ with the same right end points.
By the strong convexity of
$\overline{F}$, $s(t)$ is a strictly increasing continuous function on $[0,\tfrac{\pi}{d}]$. So $(\tfrac{\pi}2-\tfrac{(d-1)\pi}d,\tfrac{\pi}{2})$ stays in the moving interval
$(s(t)-\pi,s(t))$ with $t$ increasing, until $t$ reaches $\tfrac{\pi}{d}$,
and the two intervals have the same left end points.

This observation proves
$$g^{\overline{F}}_{c(t)}(c(t),\gamma(s))\geq 0, \quad\forall t\in[0,\tfrac{\pi}d],\ \forall s\in[\tfrac{\pi}2-\tfrac{(d-1)\pi}d,\tfrac{\pi}2],$$
and the equality happens if and only if
$$\mbox{either}\quad (t,s)=(0,\tfrac{\pi}2)\quad\mbox{or}\quad
(t,s)=(\tfrac{\pi}d,
\tfrac{\pi}d-\tfrac{\pi}2).$$
So ${\sin(t+\tfrac{k_i\pi}d)}(2f(t))^{1/2}+
\tfrac{\cos(t+\tfrac{k_i\pi}d)}{(2f(t))^{1/2}}
\tfrac{{\rm d}}{{\rm d}t}f(t)$ is a non-negative factor in $g_{ii}$ as well, and it vanishes if and only if
(\ref{0002}) happens.

Summarizing above observations and using Lemma \ref{lemma-7},
Lemma \ref{lemma-14} and
the $D_{2d}$-symmetry,
we see $g_{ii}$ is strictly positive at
$x\in S _{\overline{F}}\backslash (N_0\cup N_{\pi/d})$, and $g_{ii}$ vanishes
at $x\in S_{\overline{F}}\cap (N_0\cup N_{\pi}/d)$ if and only if $X_i$ vanishes
there. With the normal plane $\mathbf{V}$ for $M_t$ changing
arbitrarily, we see
$(g_{\alpha\beta})$
is positive definite on  $S_F\backslash(N_0\cup N_{\pi/d})$.

For $x\in N_0$, Lemma \ref{lemma-7} and Lemma \ref{lemma-14} tell us that $T_x(\mathbb{R}_{>0}N_0)=T_{x/|x|}(\mathbb{R}_{>0}M_0)\subset \mathbb{R}^n$ has a $g$-orthogonal basis consisting of $\partial_r$ and the all nonzero values of $X_i$ at $x$. So the restriction of $g$ to $T_x(\mathbb{R}_{>0}N_0)$
is positive definite. On the other hand,
for any nonzero vector $v$ in the $g$-orthogonal complement $\mathbf{V}'$ of $T_x(\mathbb{R}_{>0}N_0)$ in $T_x(\mathbb{R}^n\backslash\{0\})$, $x$ and $v$ linearly span a normal plane for $M_t$, and up to a positive scalar change,
$v$ coincides with the values at $x$ for the vector field $T$
on $\mathbf{V}\backslash\{0\}$. By Lemma \ref{lemma-7} and (\ref{0001}), we see $g(v,v)>0$. So $g$ is positive definite on
$\mathbf{V}'\subset T_x(\mathbb{R}^n\backslash\{0\})$ as well.

This argument proves that $g$ is positive definite on $N_0$. Similar
argument proves the $g$ is positive definite on $N_{\pi/d}$.
By the positive 1-homogeneity, the proof for the strong convexity is done.

To summarize, above argument proves the claim in Theorem \ref{main-thm-1} from (2) to (1).

Nextly, we prove the claim in Theorem \ref{main-thm-1} from (1) to (2), i.e., we assume $F=r\sqrt{2f(t)}$ is a Minkowski norm and prove the properties of $f(t)$ claimed in Theorem \ref{main-thm-1}. We just need to discuss the restriction of $F$ to any normal plane for $M_t$, then we easily see (1) implies (3). The equivalence between (2) and (3) has been observed in the midway, so the claim from (1) to (2) is proved.

\subsection{Foliation on $S_F$ induced by $M_t$}
\label{subsection-3-5}
In Section \ref{subsection-3-4}, we have mentioned the foliation
$N_t=S_F\cap \mathbb{R}_{>0}M_t$ on $S_F$, which is induced by $M_t$ on $S^{n-1}(1)$.
Now we prove the following theorem.

\begin{theorem}\label{main-thm-2}
Let $F=r\sqrt{2f(t)}$ be a Minkowski norm induced by the isoparametric foliation $M_t$ on the unit sphere $(S^{n-1}(1),g^{\mathrm{st}})$ and $g$ its Hessian metric. Then the foliation $N_t=S_F\cap \mathbb{R}_{>0}M_t$ on $(S_F,g)$ is isoparametric.
\end{theorem}

\begin{proof}Denote $d$ the number of principal curvatures for
the foliation $M_t$ on $(S^{n-1}(1),g^{\mathrm{st}})$.
The spherical $t$-coordinate can be viewed as a regular smooth function on the conic open subset
$C(S^{n-1}(1)\backslash(M_0\cup
M_{\pi/d})=C(S_F\backslash(N_0\cup N_{\pi/d}))$, which is still denoted as $t$. Its level sets provides the foliation $N_t$.
We will first prove the function
$t|_{S_F\backslash(N_0\cup N_{\pi/d})}$ is isoparametric on $(S_F\backslash(N_0\cup N_{\pi/d}),g)$, where $g$ is the Hessian metric of $F=r\sqrt{2f(t)}$.

Let $\{\partial_r,\partial_t,X_1,\cdots,X_{n-2}\}$ be any spherical local frame induced by $M_t$ and denote $T=\partial_t-\tfrac{r}{2f(t)}\tfrac{{\rm d}}{{\rm d}t}f(t)\partial_r$.
Since $\partial_t \cdot t=1$ and $\partial_r\cdot t=X_i\cdot t=0$, by Lemma \ref{lemma-7},
%i.e., the local frame $\{\partial_r,T,X_1,\cdots,X_{n-2}\}$
%is $g$-orthogonal and
%\begin{equation*}
%g(T,T)=\tfrac{r^2}{2f(t)}\left(2f(t)\tfrac{{\rm d}^2}{{\rm d}t^2}f(t)-\left(\tfrac{{\rm d}}{{\rm d}t}f(t)\right)^2+4f(t)^2\right),
%\end{equation*}
the gradient field $\mathrm{grad}^E t$ on $C(S_F\backslash(N_0\cup N_{\pi/d})$ is
$$\mathrm{grad}^E t=\left(\tfrac{2f(t)}{r^2\left(4f(t)^2
-\left(\tfrac{{\rm d}}{{\rm d}t}f(t)\right)^2+2f(t)
\tfrac{{\rm d}^2}{{\rm d}t^2}f(t)\right)}\right)
\ T,$$ and
its pointwise $g$-norm square is
$$g(\mathrm{grad}^E t,\mathrm{grad}^E t)
={\tfrac{2f(t)}{r^2\left(4f(t)^2
-\left(\tfrac{{\rm d}}{{\rm d}t}f(t)\right)^2+2f(t)
\tfrac{{\rm d}^2}{{\rm d}t^2}f(t)\right)}}.$$
If restricted to $S_F\backslash(N_0\cup N_{\pi/d})$, where $r^2=(2f(t))^{-1}$,
$$g(\mathrm{grad}^E t,\mathrm{grad}^E t)|_{S_F\backslash(N_0\cup N_{\pi/d})}
=\tfrac{4f(t)^2}{{4f(t)^2
-\left(\tfrac{{\rm d}}{{\rm d}t}f(t)\right)^2+2f(t)
\tfrac{{\rm d}^2}{{\rm d}t^2}f(t)}},$$
is a function of $t$.

Denote $\mathrm{Hess}(\cdot,\cdot)$ the Hessian with respect to  $g$ on $\mathbb{R}^n\backslash\{0\}$, then we have
$$\mathrm{Hess}(X,X)t=X\cdot (X\cdot t)-(\nabla_XX)\cdot t=-(\nabla_XX)\cdot t,$$
for each $X\in\{\partial_r,T,X_1,\cdots,X_{n-2}\}$. Here $\nabla$ is
the Levi-Civita connection of $(\mathbb{R}^n\backslash\{0\},g)$.
So the Laplacian $\Delta^E t$ on $(C(S_F\backslash(N_0\cup N_{\pi/d}), g)$ can be presented as
\begin{eqnarray}
\Delta^E t&=&
\tfrac{\mathrm{Hess}(\partial_r,\partial_r)t}{g(\partial_r,\partial_r)}
+\tfrac{\mathrm{Hess}(T,T)t}{g(T,T)}+
\sum_{i=1}^{n-2}
\tfrac{\mathrm{Hess}(X_i,X_i)t}{g(X_i,X_i)}
\nonumber\\
&=&
-\tfrac{(\nabla_{\partial_r}{\partial_r})\cdot
t}{g(\partial_r,\partial_r)}-\tfrac{(\nabla_TT) \cdot t}{g(T,T)}
-\sum_{i=1}^{n-2}\tfrac{(\nabla_{X_i}X_i)\cdot t}{g(X_i,X_i)},\label{0003}
\end{eqnarray}
where the spherical local frame is defined.

Using Lemma \ref{lemma-3}, we collect the following information
for $\nabla$:
\begin{eqnarray}
\nabla_{\partial_r}\partial_r&=&0,\label{0004}\\
\nabla_TT&=&\tfrac{\left(\tfrac{{\rm d}}{{\rm d}t}f(t)\right)^3-2f(t)\tfrac{{\rm d}}{{\rm d}t}f(t)\tfrac{{\rm d}^2}{{\rm d}t^2}f(t)+f(t)^2\tfrac{{\rm d}^3}{{\rm d}t^3}f(t)}{4f(t)^3-f(t)\left(\tfrac{{\rm d}}{{\rm d}t^2}f(t)\right)^2+2f(t)^2\tfrac{{\rm d}^2}{{\rm d}t^2}f(t)}\ \partial_t\ \mbox{(mod}\ \partial_r,X_1,\cdots,X_{n-2}\mbox{)},\label{0005}\\
\nabla_{X_i}X_i &\equiv& \tfrac{\tfrac{{\rm d}}{{\rm d}t}f_i(t)\left(\tfrac{{\rm d}}{{\rm d}t}f(t)\right)^2-4f(t)^2\tfrac{{\rm d}}{{\rm d}t}f_i(t)-f(t)\tfrac{{\rm d}}{{\rm d}t}f_i(t)\tfrac{{\rm d}^2}{{\rm d}t^2}f(t)-f(t)\tfrac{{\rm d}}{{\rm d}t}f(t)\tfrac{{\rm d}^2}{{\rm d}t^2}f_i(t)}{8f(t)^2-2\left(\tfrac{{\rm d}}{{\rm d}t}f(t)\right)^2+4f(t)\tfrac{{\rm d}^2}{{\rm d}t^2}f(t)}\ \partial_t
\nonumber\\
& &\mbox{(mod}\ \partial_r,X_1,\cdots,X_{n-2}\mbox{)}.\label{0006}
\end{eqnarray}

 Input the formula of $g(X_i,X_i)$ in Lemma \ref{lemma-3},  (\ref{0007}) in Lemma \ref{lemma-7}, and (\ref{0004})-(\ref{0006}), into (\ref{0003}),  we see that $\Delta^E t$ is the product of $r^{-2}$ and a function of $t$. In particular, its restriction to $S_F$ (where $r^{-2}=2f(t)$) only depends on the values of $t$.

Finally, we consider the gradient $\mathrm{grad}^S (t|_{S^F\backslash(N_0\cup N_{\pi/d})})$ and the
Laplacian $\Delta^S (t|_{S^F\backslash(N_0\cup N_{\pi/d})})$
on $(S^F\backslash(N_0\cup N_{\pi/d}), g)$. Since the function $t$ is constant along each ray initiating from the origin, $\mathrm{grad}^E t$
is tangent to $S_F$, so we have
$$\mathrm{grad}^S (t|_{S^F\backslash(N_0\cup N_{\pi/d})})=
(\mathrm{grad}^E t) |_{S^F\backslash(N_0\cup N_{\pi/d})}.$$
By (14.3.10) in \cite{BCS2000}, we also have
\begin{eqnarray*}
\Delta^S (t|_{S^F\backslash(N_0\cup N_{\pi/d})})=
(\Delta^E t)|_{S^F\backslash(N_0\cup N_{\pi/d})}.
\end{eqnarray*}
So both $g(\mathrm{grad}^S (t|_{S^F\backslash(N_0\cup N_{\pi/d})}), \mathrm{grad}^S (t|_{S^F\backslash(N_0\cup N_{\pi/d})}))$ and $\Delta^S (t|_{S^F\backslash(N_0\cup N_{\pi/d})})$ are functions on $S_{F}\backslash(N_0\cup N_{\pi/d})$ which only depend on $t$.

To summarize, above argument proves that $t$ is isoparametric on $(S_F\backslash(N_0\cup N_{\pi/d}),g)$. Then we can
choose $\psi(s)$ on $[0,\tfrac{\pi}{d}]$ which satisfies $\tfrac{{\rm d}}{{\rm d}t}\psi(s)>0$ on $(0,\tfrac{\pi}d)$ and can be extended to a smooth function on $\mathbb{R}$ satisfying $\psi(s)=\psi(-s)$ and $\psi(s)=\psi(\tfrac{2\pi}{d}-s)$. Then the composition $\psi\circ (t|_{S_F\backslash(N_0\cup N_{\pi/d})})$ can be extended to a smooth function $p(\cdot)$ on $(S_F,g)$, which level sets provide the foliation $N_t$.
It is easy to verify the isoparametric property of $p(\cdot)$ from that of
$t|_{S_F\backslash(N_0\cup N_{\pi/d})}$. So the foliation
$N_t$ on $(S_F,g|_{S_F})$ is isoparametric.
\end{proof}

By Lemma \ref{lemma-7} and Theorem \ref{main-thm-2}, we see $T=\partial_t-\tfrac{r}{2f(t)}\tfrac{{\rm d}}{{\rm d}t}f(t)\partial_r$ generates the normal geodesics for $N_t$ on $(S_F,g)$, which
have the spherical coordinates presentations $t\mapsto ((2f(t))^{-1/2},t,\xi)$ for any fixed $\xi\in M_0$. So we
have the following corollary.
\begin{corollary}\label{cor-3}
Let $M_t$ be an isoparametric foliation on  $(S^{n-1}(1),g^{\mathrm{st}})$, for which we have induced Minkowski norm $F=r\sqrt{2f(t)}$ with the Hessian metric $g$, and the isoparametric foliation
$N_t=S_F\cap \mathbb{R}_{>0}M_t$ on $(S_F, g)$. Then the intersection with $S_F$ provides
a one-to-one correspondence between the set of all normal planes
for $M_t$ and the set of all unparametrized maximally extended
normal geodesics for $N_t$ in $(S_F,g)$.
\end{corollary}
%\begin{eqnarray*}
%\nabla^E t&=&
%\tfrac{\mathrm{Hess(\partial,\partial)t}}{g(\partial_r,\partial)}
%+\tfrac{\mathrm{Hess(T,T)t}}{g(T,T)}
%+\tfrac{\mathrm{Hess()}}
%\end{eqnarray*}

\section{Hessian isometry and Laugwitz conjecture}

\subsection{Hessian isometry and local Hessian isometry}
\label{subsection-4-1}
Let $F_1$ and $F_2$ be two Minkowski norms on $\mathbb{R}^n$
with $n\geq2$
and denote $g_1=g_1(\cdot,\cdot)$ and $g_2=g_2(\cdot,\cdot)$ their Hessian metrics respectively.
A diffeomorphism $\Phi$ on $\mathbb{R}^n\backslash\{0\}$ is called
a {\it Hessian isometry} from $F_1$ to $F_2$, if it is an isometry from $g_1$ to $g_2$.

Since the rays initiating from the origin provide the set of all incomplete geodesics on $(\mathbb{R}^n\backslash\{0\},g_i)$, and
the Hessian metric $g_i$ on $\mathbb{R}^n\backslash\{0\}$
has the presentation $g_i=({\rm d}F_i)^2+F_i^2 (g_i|_{S_{F_i}})$,  we have the following easy lemma.
\begin{lemma}\label{lemma-4}
Any Hessian isometry $\Phi$ from $F_1$ to $F_2$ maps the indicatrix $S_{F_1}$ to the indicatrix $S_{F_2}$, and it is positively 1-homogeneous, i.e.,
$\Phi(\lambda x)=\lambda\Phi(x)$ for any $\lambda>0$ and any $x\in\mathbb{R}^n\backslash\{0\}$. Conversely, any isometry between $(S_{F_i},g_i)$
can be uniquely extended to a Hessian isometry from $F_1$ to $F_2$.
\end{lemma}

Besides the global Hessian isometry, a {\it local Hessian isometry}
can be defined as an isometric diffeomorphism between two conic open subsets in $(\mathbb{R}^n\backslash\{0\},g_i)$ respectively, satisfying the positive 1-homogeneity. An analog of Lemma
\ref{lemma-4} is valid for local Hessian isometries, i.e., the restriction to indicatrix provides a one-to-one correspondence between local Hessian isometries from $F_1$ to $F_2$ and local isometries from $(S_{F_1},g_1)$ to $(S_{F_2},g_2)$.

\subsection{Linear isometry and two applications}
\label{subsection-4-2}
Any linear isomorphism $\Phi:(\mathbb{R}^n,F_1)\rightarrow(\mathbb{R}^n,F_2)$ which
satisfies $F_1=F_2\circ\Phi$ naturally induces a Hessian isometry
when it is restricted to $\mathbb{R}^n\backslash\{0\}$. For simplicity,
we call it a {\it linear isometry}.

Here
we propose two applications of linear isometries for
 the Minkowski norm $F=r\sqrt{2f(t)}$ induced by an isoparametric foliation $M_t$ on $(S^{n-1}(1),g^{\mathrm{st}})$.

Firstly, when the foliation $M_t$ on $(S^{n-1}(1),g^{\mathrm{st}})$ is homogeneous,
i.e., there exists a compact connected Lie subgroup $G$ of $SO(n)$, such that each $M_t$ is a $G$-orbit,
then the following lemma shows us a shortcut to Theorem \ref{main-thm-2}.

\begin{lemma}\label{lemma-5}
Let $M_t$ be a homogeneous isoparametric foliation on $(S^{n-1}(1),g^{\mathrm{st}}  )$, and $F=r\sqrt{2f(t)}$ a Minkowski norm induced by $M_t$. Then the foliation $N_t=S_F\cap \mathbb{R}_{>0}M_t$ on $(S_F,g)$
is also a homogeneous isoparametric foliation.
\end{lemma}

\begin{proof}Let $G$ be the compact connected Lie subgroup of $SO(n)$ such that each $M_t$ is a $G$-orbit. Then the induced Minkowski norm $F=r\sqrt{2f(t)}$ is $G$-invariant. So the $G$-action
on $(\mathbb{R}^n,F)$ is linearly isometric and it is of cohomogeneity
one when restricted to $(S^F,g)$. Each $N_t=S_F\cap
\mathbb{R}_{>0}M_t$ is $G$-orbit, and singular $G$-orbits only appear at the two ends, i.e., $N_0$ and $N_{{\pi}/d}$. By the theory of Riemannian manifold of cohomogeneity one, the corresponding isoparametric function can be constructed.
\end{proof}

Theorem \ref{main-thm-2} is a direct corollary of
Lemma \ref{lemma-5}
when we have $d\in\{1,2,3,6\}$ for the principal curvatures of $M_t$, and for some
subcases with $d=4$.

Secondly, we prove Laugwitz Conjecture in a special case.
Laugwitz conjectured that if the Hessian metric $g$ of a Minkowski norm $F$ is flat on $\mathbb{R}^n\backslash\{0\}$ with $n\geq 3$,
then $F$ is Euclidean \cite{La1965}. When $F$ is reversible, i.e., $F(x)=F(-x)$, $\forall x\in\mathbb{R}^n$, or equivalently, when $F$ is absolutely 1-homogeneous, i.e., $F(\lambda x)=|\lambda| F(x)$, $\forall \lambda\in\mathbb{R},x\in\mathbb{R}^n$, F. Brickell proved this conjecture by the following theorem \cite{Br1967,Sc1968}.
\begin{theorem} \label{cite-theorem-Brickell}
Let $F$ be a Minkowski norm on $\mathbb{R}^n$ with $n\geq3$, satisfying the
reversibility condition, i.e., $F(x)=F(-x)$, $\forall x\in\mathbb{R}^n$.
If its Hessian metric $g$ is flat on $\mathbb{R}^n\backslash\{0\}$, then $F$ is
Euclidean.
\end{theorem}
Recently, we proved the following theorem in \cite{XM2020}.
\begin{theorem}\label{cite-theorem-XM-2020}
Laugwitz conjecture is true for the class of Minkowski norms
which are invariant with respect to the standard block diagonal
$SO(n-1)$-action.
\end{theorem}
Now, we further strengthen it as following.
\begin{theorem}\label{main-thm-3}
Laugwitz conjecture is true for Minkowski norms induced
by an isoparametric foliation on the unit sphere.
\end{theorem}
\begin{proof}Let $M_t$ be an isoparametric foliation on
$(S^{n-1}(1),g^{\mathrm{st}})$ with $n>2$ and $d$ the number of principal curvatures of $M_t$. Suppose $F=r\sqrt{2f(t)}$ is
a Minkowski norm induced by $M_t$, with a flat Hessian metric $g$
on $\mathbb{R}^n\backslash\{0\}$.

{\bf Case 1}. We have $d\in\{2,4,6\}$. In this case, we only need to prove that $F$ is reversible, then Theorem \ref{main-thm-3} follows from Theorem \ref{cite-theorem-Brickell} immediately.

To prove our claim, we consider any $x\in S^{n-1}(1)$. Then
there always exists a normal plane $\mathbf{V}$ for $M_t$, which contains $x$. The antipodal map on $\mathbf{V}$ is contained in $D_{2d}$ when $d$ is even. So the $D_{2d}$-invariancy of $\overline{F}=F|_{\mathbf{V}}$ implies $F(x)=F(-x)$.

{\bf Case 2}. We have $d=1$.
This case has already been proved by Theorem \ref{cite-theorem-XM-2020}.

{\bf Case 3}. We have $d=3$. In this case, E. Cartan found the following explicit construction \cite{Ca1939,Ca1940}. We present $\mathbb{R}^n$ with $n\in\{5,8,14,26\}$ as
$$\mathbb{R}^n=\mathbb{R}^2\oplus\mathbb{F}^3=\{(a,b,x,y,z)| \forall a,b\in\mathbb{R},\ x,y,z\in\mathbb{F}\}$$
with $\mathbb{F}\in\{\mathbb{R},\mathbb{C},\mathbb{H},\mathbb{O}\}$
respectively, such that
the standard Euclidean norm is given by
$$|(a,b,x,y,z)|=\sqrt{a^2+b^2+x\overline{x}+
y\overline{y}+z\overline{z}}.$$
Then isoparametric function for the foliation $M_t$ can be chosen as
\begin{eqnarray*}
p&=&a^3-3ab^2+
\tfrac{3a}{2}(x\overline{x}+y\overline{y}-2z\overline{z})
+\tfrac{3\sqrt{3}b}{2}(x\overline{x}-y\overline{y})
+\tfrac{3\sqrt{3}}{2}\left((xy)z+\overline{(xy)z}\right),
\end{eqnarray*}
where $a^2+b^2+x\overline{x}+y\overline{y}+z\overline{z}=1$.

Both the standard Euclidean norm and the function $p(\cdot)$ are invariant for the action of  $\Phi(a,b,x,y,z)=(a,b,-x,-y,z)$, So $\Phi$ is a linear isometry on $(\mathbb{R}^n,F)$. The fixed point set of $\Phi$ is
the Euclidean subspace $\mathbb{R}^{n'}=\{(a,b,0,0,z)| \forall a,b\in\mathbb{R}, z\in\mathbb{F}\}$ with dimension $n'=3,4,6,10$
when $n=5,8,14,26$ respectively. Restricted to the unit sphere $S^{n'-1}(1)=\mathbb{R}^{n'}\cap S^{n-1}(1)$, where we have $a^2+b^2+z\overline{z}=1$ and $x=y=0$,
the function $p(\cdot)$ is then given by
$$p|_{S^{n'-1}(1)}
=a^3-3ab^2-3az\overline{z}=4a^3-3a.$$
So the connected components of all $M'_t=\mathbb{R}^{n'}\cap M_t$ provide a homogeneous isoparametric foliation
on the unit sphere $(S^{n'-1}(1), g^{\mathrm{st}})$
induced by a standard block diagonal $SO(n'-1)$-action.

Denote the restriction $F'=F|_{\mathbb{R}^{n'}}$. Then $F'$  is invariant
with respect to  a standard block diagonal action of $SO(n'-1)$.
The Hessian metric $g'$ of $F'$ coincides with the restriction $g|_{\mathbb{R}^{n'}\backslash\{0\}}$.
Since $\mathbb{R}^{n'}\backslash\{0\}$ is the fixed point set of the
linear isometry $\Phi$ on $(\mathbb{R}^n\backslash\{0\},g)$, it is
totally geodesic. Because $g$ is flat on $\mathbb{R}^n\backslash\{0\}$, $g'$ is also flat on $\mathbb{R}^{n'}\backslash\{0\}$.
By Theorem \ref{cite-theorem-XM-2020}, $F'$ is Euclidean.

If we use the spherical coordinates $(r,t,\xi)\in\mathbb{R}_{>0}\times
(0,\pi)\times S^{n'-2}(1)$ on $\mathbb{R}^{n'}$, such that
$$a=r\cos t\quad \mbox{and}\quad (b,z)=r\sin t\ \xi,$$
then $F'$ has the presentation $F'=r\sqrt{2f(t)}$, where
$f(t)=c_1+c_2\cos 2t$ for some constants $c_1$ and $c_2$ with $c_1>|c_2|$. Meanwhile, $f(t)$ only depends on the values of
$$p|_{S^{n'-1}(1)}=4a^3-3a=4\cos^3 t-3\cos t=\cos 3t,$$
which has a zero derivative at $t=\tfrac{\pi}3$. So we have $\tfrac{{\rm d}}{{\rm d}}f(t)|_{t=\pi/3}=0$, i.e.,
$c_2=0$ and $f(t)$ is a positive constant function.

Since the same $f(t)$ is also used in the presentation $F=r\sqrt{2f(t)}$, we see $F$ must be Euclidean as well. The case $d=3$ is proved.
\end{proof}
\section{Hessian isometry which preserves the orientation and fixes the spherical $\xi$-coordinates}
\subsection{Notations for the spherical coordinates presentation}
\label{subsection-5-1}
Let $M_t$ be an isoparametric foliation on
$(S^{n-1}(1),g^{\mathrm{st}})$ with $d$ principal curvatures. When we mention the spherical coordinates, spherical local frame, induced Minkowski norm, etc., they are always referred to $M_t$. Let $F_1$ and $F_2$ be two induced Minkowski norms on $\mathbb{R}^n$. We denote their indicatrices and Hessian metrics as $S_{F_i}$ and $g_i$ respectively. On each $(S_{F_i},g_i)$, we have the induced isoparametric foliation $N_{i,t}=S_{F_i}\cap \mathbb{R}_{>0}M_t$.
To distinguish the two Minkowski norms, we use $\theta$ to denote the spherical $t$-coordinate for $F_2$. So we have the spherical coordinates presentations
\begin{equation}\label{0008}
F_1=r\sqrt{2f(t)}\quad\mbox{and}\quad F_2=r\sqrt{2h(\theta)}.
\end{equation}
Though (\ref{0008}) only use the values of $f(t)$ and $h(\theta)$ on $[0,\tfrac{\pi}{d}]$, as Theorem \ref{main-thm-1} indicates,
both functions can and will be extended to $D_{2d}$-invariant positive smooth functions on $\mathbb{R}/(2\mathbb{Z}\pi)$, and they satisfy the inequality in Theorem \ref{main-thm-1}.

Let $\mathbf{V}$ be any normal plane for $M_t$. We parametrize $\mathbf{V}\cap S^{n-1}(1)$ as $\gamma(t)$ with its $g^{\mathrm{st}}$-arc length parameter $t\in\mathbb{R}/(2\mathbb{Z}\pi)$ and
$\gamma(0)\in M_0$. The normal plane
$\mathbf{V}$ can be identified with $\mathbb{R}^2$ as indicated in Section \ref{subsection-2-4}, so that $v_1=\gamma(0)$ and $v_2=\gamma(\tfrac{\pi}2)$ are mapped to $e_1=(1,0)$ and $e_2=(0,1)$ respectively. The polar coordinates $(r,t)\in\mathbb{R}_{>0}\times(\mathbb{R}/(2\mathbb{Z}\pi))$ on $\mathbf{V}$ is determined by $$x=x_1v_1+x_2v_2=
r\cos t\ v_1+r\sin t\ v_2$$
for any $x\in\mathbf{V}\backslash\{0\}$. Each interval $(\tfrac{k\pi}{d},\tfrac{(k+1)\pi}{d})$ with $k\in\{0,\cdots,2d-1\}$
for the polar $t$-coordinate determines a conic open subset of $\mathbf{V}\backslash\{0\}$, which corresponds to a distinct spherical $\xi$-coordinate on $\mathbb{R}^n$.
%Then the polar coordinates $(r,t)\in\mathbb{R}_{>0}\times (\mathbb{R}/(2\mathbb{Z}\pi))$ on $\mathbb{R}^2$ with
%$x_1=r\cos t$ and $x_2=r\sin t$, and the action of the dihedral group
%$D_{2d}$ generated by left matrix multiplication by $$\left(
%                                                       \begin{array}{cc}
%                                                         1 & 0 \\
%                                                         0 & -1 \\
%                                                       \end{array}
%                                                     \right)
%\quad\mbox{and}\quad
%\left(
%         \begin{array}{cc}
%           \cos(\tfrac{2\pi}{d}) & -\sin(\tfrac{2\pi}{d})   \\
%           \sin(\tfrac{2\pi}d) &  \cos(\tfrac{2\pi}{d})\\
%         \end{array}
%       \right),
%$$
%which fix $r$ and change $t$ to $-t$ and $t+\tfrac{2\pi}{d}$ respectively, are naturally inherited by $\mathbf{V}$. Using this convention, the restrictions
The restrictions $\overline{F}_i=F_i|_{\mathbf{V}}$ have the polar
coordinates presentations
$\overline{F}_1=r\sqrt{2f(t)}$ and $\overline{F}_2=r\sqrt{2h(\theta)}$ (similarly we use $\theta$ to denote the polar $t$-coordinates for $\overline{F}_2$), where $f(t)$ and $h(\theta)$ are exactly those in (\ref{0008}) after extension.

In this section, we discuss a Hessian isometry $\Phi$ from $F_1$ to $F_2$ which {\it preserves the orientation and fixes the spherical $\xi$-coordiantes}, i.e., it satisfies the following conditions:
\begin{enumerate}
\item $\Phi$ is an orientation preserving diffeomorphism on $\mathbb{R}^n\backslash\{0\}$.
\item $\Phi$ preserves the conic open subset $C(S^{n-1}(1)\backslash M_0\cup M_{\pi/d})=\mathbb{R}^n\backslash(\mathbb{R}_{\geq0}M_0\cup
    \mathbb{R}_{\geq0}M_{{\pi}/{d}})$, and
for any $x\in C(S^{n-1}(1)\backslash M_0\cup M_{\pi/d})$,
$x$ and $\Phi(x)$ has the same $\xi$-coordinates.
\end{enumerate}

By Lemma \ref{lemma-4}, we have $\Phi(S_{F_1})=S_{F_2}$. Since the condition (2) requires that $\Phi$ fixes the spherical $\xi$-coordinates,
our previous observation indicates $\Phi$ preserves each arbitrarily chosen normal plane $\mathbf{V}$, i.e., $\overline{\Phi}=\Phi|_{\mathbf{V}}$ is a Hessian isometry from $\overline{F}_1=F_1|_{\mathbf{V}}$ to $\overline{F}_2=F_2|_{\mathbf{V}}$.
Further more, $\Phi$ fixes each point in $\mathbf{V}\cap (M_0\cup M_{\pi/2d})$ and preserves each conic open subset in $\mathbf{V}\backslash\{0\}$
with polar $t$-coordinate in $(\tfrac{k\pi}{d},\tfrac{(k+1)\pi}d)$.

More discussion for the restriction to $\mathbf{V}$ is postponed to Section \ref{subsection-5-3}. Here we only need to apply
Corollary \ref{cor-3} to notice that $\Phi$ maps each normal geodesic for the isoparametric foliation $N_{1,t}=S_{F_1}\cap \mathbb{R}_{>0}M_t$ on $(S_{F_1},g_1)$ to that for $N_{2,t}=S_{F_2}\cap\mathbb{R}_{>0} M_t$
on $(S_{F_2},g_2)$. Meanwhile, $\Phi$ preserves the foliation $\mathbb{R}_{>0}M_t$ on $\mathbb{R}^n\backslash\{0\}$.
%Further more, $\Phi$ is required to fix the spherical $\xi$-coordinates, so it preserves the normal plane $\mathbf{V}$, and the restriction $\overline{\Phi}=\Phi|_{\mathbf{V}\backslash\{0\}}$ is a Hessian isometry between $\overline{F}_i=F_i|_{\mathbf{V}}$.
%We fix a normal plane $\mathbf{V}$
%$\Phi|_{\mathbf{V}}$ is a Hessian isometry between the Hessian metrics
%$g^{F'_i}=g^{F_i}|_{\mathbf{V}\backslash\{0\}}$.
So $\Phi$ maps each $N_{1,t}=S_{F_1}\cap\mathbb{R}_{>0}M_t$ to some $N_{2,\theta(t)}=S_{F_2}\cap\mathbb{R}_{>0}M_{\theta(t)}$.
To summarize, $\Phi$ has the following spherical coordinates presentation
\begin{equation*}
(r,t,\xi)\mapsto (\tfrac{r f(t)^{1/2}}{h(\theta(t))^{1/2}},\theta(t),\xi).
\end{equation*}
Since $\Phi$ is an orientation preserving diffeomorphism by the condition (1), it can be observed immediately from the Jacobi matrix
$$\left(
    \begin{array}{ccc}
      \tfrac{f(t)^{1/2}}{h(\theta(t))^{1/2}} &
      \tfrac{h(\theta(t))-f(t)\tfrac{{\rm d}}{{\rm d}t}\theta(t)}{2f(t)(2h(\theta(t)))^{3/2}} & 0 \\
      0 & \tfrac{{\rm d}}{{\rm d}t}\theta(t) & 0 \\
      0 & 0 & \mathrm{Id} \\
    \end{array}
  \right).
$$
for the tangent map  $\Phi_*$ that $\theta(t)$ is a diffeomorphism on $(0,\tfrac{\pi}d)$ with positive derivative everywhere, and $\theta(t)=t$ for $t\in\{0,\tfrac{\pi}{d}\}$ by continuity. Indeed, $\theta(t)$ can and will be extended an orientation preserving diffeomorphism on $\mathbb{R}/(2\mathbb{Z}\pi)$ with a $D_{2d}$-equivariancy, i.e., $\theta(-t)=-\theta(t)$ and
$\theta(t+\tfrac{2\pi}{d})=\theta(t)+\tfrac{2\pi}{d}$, which will appear
in the polar coordinates presentation
$(r,t)\rightarrow (\tfrac{r f(t)^{1/2}}{h(\theta(t))^{1/2}},\theta(t))$
for $\overline{\Phi}=\Phi|_{\mathbf{V}}$. After the extension, $\theta(t)$ fixes each point in
$\tfrac{\mathbb{Z}\pi}d\subset\mathbb{R}/(2\mathbb{Z}\pi)$.

Summarizing above argument, we get the following lemma for the
spherical coordinates presentation for $\Phi$.

\begin{lemma}\label{lemma-15}
Let $F_1=r\sqrt{2f(t)}$ and $F_2=r\sqrt{2h(\theta)}$ be two Minkowski norms induced by $M_t$. Then any
Hessian isometry $\Phi$ from $F_1$ to $F_2$ which preserves the orientation and fixes the spherical $\xi$-coordinates has the spherical
coordinates presentation
\begin{equation}\label{0009}
(r,t,\xi)\mapsto (\tfrac{r f(t)^{1/2}}{h(\theta(t))^{1/2}},\theta(t),\xi),
\end{equation}
in  which $\theta(t)$ is a $D_{2d}$-equivariant orientation preserving diffeomorphism on
$\mathbb{R}/(2\mathbb{Z}\pi)$ and fixes each point in $\tfrac{\mathbb{Z}\pi}{d}$.
\end{lemma}

\subsection{Description by an ODE system}
Nextly, we consider a spherical local frame
$\{\partial_r$, $\partial_t$ (or $\partial_\theta$ for $F_2$), $X_1,\cdots,X_{n-2}\}$, defined in the conic open subset of $\mathbb{R}^n\backslash\{0\}$ which
only requires the spherical $\xi$-coordinate to be contained in some
open subset in $M_{\pi/2d}$, i.e.,  its defining domain is preserved by $\Phi$. We
denote $T_1=\partial_{t}-\tfrac{r}{2f(t)}\tfrac{{\rm d}}{{\rm d}t}f(t)\partial_r$ and
$T_2=\partial_{\theta}-\tfrac{r}{2h(\theta)}\tfrac{{\rm d}}{{\rm d}\theta}h(\theta)\partial_r$ which are tangent to $S_{F_1}$ and $S_{F_2}$ respectively. Then we have
\begin{lemma}\label{lemma-6}
The tangent map
$\Phi_*$ for $\Phi$ satisfies
\begin{eqnarray}
\Phi_*(T_1)&=&\tfrac{{\rm d}}{{\rm d}t}\theta(t)\  T_2,\quad\mbox{and}\label{0010}\\
\Phi_*(X_i)&=&X_i,\quad \forall 1\leq i\leq n-2.\label{0011}
\end{eqnarray}
\end{lemma}

\begin{proof} We first prove (\ref{0010}).
By the positive 1-homogeneity (i.e., Lemma \ref{lemma-4}), we may restrict our discussion to $S_{F_1}$.
On $(S_{F_1},g_1)$, $T_1$ generates a normal geodesic $c_1(t)$ with $t\in(0,\tfrac{\pi}{d})$ for the isoparametric foliation $N_{1,t}$. Its image $\Phi(c_1(t))$ coincides with a normal geodesic $c_2(\theta)$ with $\theta\in(0,\tfrac{\pi}{d})$ for the isoparametric foliation $N_{2,\theta}$ on
$(S_{F_2},g_2)$, i.e., an integral curve of $T_2$, up to a change of parameter $\theta=\theta(t)$. So we have
$\Phi_*(T_1)=\tfrac{{\rm d}}{{\rm d}t}\theta(t)\  T_2$.

Then we prove (\ref{0011}).
Let $s\mapsto ((2f(t))^{-1/2},t,\xi(s))$ with any fixed $t\in(0,\tfrac{\pi}{d})$ be the spherical coordinates presentation for
an integral curve of $X_i$ on $S_{F_1}$, then its $\Phi$-image has the spherical coordinates presentation $s\mapsto ((2h(\theta(t)))^{-1/2},\theta(t),\xi(s))$, which is still
an integral curve of $X_i$. So we have $\Phi_*(X_i)=X_i$, $\forall 1\leq i\leq n-2$.
\end{proof}

Since $\Phi$ is a Hessian isometry, (\ref{0010}) implies
\begin{equation}\label{0012}
g_1(T_1,T_1)=g_2(\Phi_*(T_1),\Phi_*(T_1))=\left(\tfrac{{\rm d}}{{\rm d}t}\theta(t)\right)^2\  g_2(T_2,T_2),
\end{equation}
where the left side is evaluated at $x=((2f(t))^{-1/2},t,\xi)\in S_{F_1}\backslash(N_{1,0}\cup N_{1,\pi/d})$
and the right side is evaluated at $\Phi(x)=((2h(\theta(t)))^{-1/2},\theta(t),\xi)\in S_{F_2}\backslash(N_{2,0}\cup N_{2,\pi/d})$.
Using (\ref{0007}) in Lemma \ref{lemma-7}, we get the following
ODE,
\begin{eqnarray}
& &\tfrac{1}{2f(t)}\tfrac{{\rm d}^2}{{\rm d}t^2}f(t)-\tfrac{1}{4f(t)^2}\left(\tfrac{{\rm d}}{{\rm d}t}f(t)\right)^2+1\nonumber\\
\label{ODE-1}&=&\left(\tfrac{{\rm d}}{{\rm d}t}\theta(t)\right)^2\
\left(\tfrac{1}{2h(\theta(t))}\tfrac{{\rm d}^2}{{\rm d}\theta^2}h(\theta(t))-\tfrac{1}{4h(\theta(t))^2}\left(\tfrac{{\rm d}}{{\rm d}\theta}h(\theta(t))\right)^2+1\right),
\end{eqnarray}
for $t\in(0,\tfrac{\pi}{d})$. By continuity and $D_{2d}$-symmetry, (\ref{ODE-1}) for $t\in(0,\tfrac{\pi}d)$ is equivalent to that for
all $t\in\mathbb{R}/(2\mathbb{Z}\pi)$.

Similarly, (\ref{0011}) implies $g_1(X_i,X_i)=g_2(X_i,X_i)$ for each $i$
%, where the left inside is evaluated at $x=((2f(t))^{-1/2},t,\xi)\in S_{F_1}\backslash(N_{1,0}\cup N_{1,\pi/d})$ and the right
%side is evaluated at $\Phi(x)=((2h(\theta(t)))^{-1/2},\theta(t),\xi)\in S_{F_2}\backslash(N_{2,0}\cup N_{2,\pi/d})$
.
Using Lemma \ref{lemma-3}, we get the following ODEs,
\begin{equation}
\label{0014}
f_i(t)+\tfrac{1}{4f(t)}\tfrac{{\rm d}}{{\rm d}t}f_i(t)\tfrac{{\rm d}}{{\rm d}t}f(t)=f_i(\theta(t))+
\tfrac{1}{4h(\theta(t))}\tfrac{{\rm d}}{{\rm d}\theta}f_i(\theta(t))\tfrac{{\rm d}}{{\rm d}\theta}h(\theta(t)),
\end{equation}
for any $1\leq i\leq n-2$ and any $t\in(0,\tfrac{\pi}{d})$. Here
$f_i(t)=a_i\sin^2(t+\tfrac{k_i\pi}{d})$ with constants $a_i>0$ and $k_i\in\{0,\cdots,d-1\}$. Because $\{k_1,\cdots,k_{n-2}\}=\{0,\cdots,d-1\}$, we can reorganize
(\ref{0014}) as
\begin{eqnarray}
& &\sin^2 (t+\tfrac{k\pi}{d})+\tfrac{\cos(t+\tfrac{k\pi}{d})
\sin(t+\tfrac{k\pi}{d})}{2f(t)}\tfrac{{\rm d}}{{\rm d}t}f(t)\nonumber\\
&=&\sin^2(\theta(t)+\tfrac{k\pi}{d})+\tfrac{\cos(\theta(t)+\tfrac{k\pi}{d})
\sin(\theta(t)+\tfrac{k\pi}{d})}{2h(\theta(t))}
\tfrac{{\rm d}}{{\rm d}\theta}h(\theta(t)),\label{0013}
\end{eqnarray}
for every $t\in(0,\tfrac{\pi}{d})$ and every $k\in\{0,\cdots,d-1\}$.

Obviously, the ODEs in (\ref{0013}) are satisfied with all $k\in\mathbb{Z}$. Further more, they are satisfied for all $t\in\mathbb{R}/(2\mathbb{Z}\pi)$ as well.
To prove this claim, we first observe that by continuity, (\ref{0013}) for each $k$ is valid at $t=0$ and $t=\tfrac{\pi}{d}$.
Using the properties $f(t)=f(-t)$,
$h(\theta)=h(-\theta)$ and $\theta(-t)=-\theta(t)$, (\ref{0013}) with $k=k'$ for $t\in[-\tfrac{\pi}{d},0]$ can be deduced from (\ref{0013})
with $k= d-k'$   for $t\in[0,\tfrac{\pi}d]$. Then using the symmetry with respect to  $\mathbb{Z}_{g}\subset D_g$, i.e., the properties $f(t)=f(t+\tfrac{2\pi}{d})$, $h(\theta)=h(\theta+\tfrac{2\pi}{d})$
and $\theta(t+\tfrac{2\pi}{d})=\theta(t)+\tfrac{2\pi}{d}$,
(\ref{0013}) with $k=k'$ for   $t\in[\tfrac{(2k''-1)\pi}{d},\tfrac{(2k''+1)\pi}d]$ can be deduced from (\ref{0013}) with $k= k'+2k''$   for $t\in[-\tfrac{\pi}d,\tfrac{\pi}d]$.

Above argument tells us how to determine the triple $(f(t),h(\theta),\theta(t))$ from $(F_1,F_2,\Phi)$, and more
importantly, list the properties the triple $(f(t),h(\theta),\theta(t))$
must satisfy. Then we observe how to use the data $(f(t),h(\theta),\theta(t))$ with those properties to construct
the wanted $(F_1,F_2,\Phi)$.

When the positive smooth $D_{2d}$-invariant functions $f(t)$ and
$h(\theta)$ satisfy the inequality in Theorem \ref{main-thm-1}, we can use them construct the induced Minkowski norms $F_1=r\sqrt{2f(t)}$ and $F_2=r\sqrt{2h(\theta)}$.

When the $D_{2d}$-equivariant orientation preserving
diffeomorphism $\theta(t)$ on $\mathbb{R}/(2\mathbb{Z}\pi)$
fixes each point in $\tfrac{\mathbb{Z}\pi}d$, we can use it to construct
a  diffeomorphism $\Phi$ on $C(S^{n-1}(1)\backslash(M_0\cup M_{\pi/d}))$ with the spherical coordinates presentation
(\ref{0009}), i.e., $(r,t,\xi)\mapsto (\tfrac{rf(t)^{1/2}}{h(\theta(t))^{1/2}},\theta(t),\xi)$. Then obviously
$\Phi$ preserves the orientation and fixes the spherical $\xi$-coordinates.
To see it can be extended to a Hessian isometry, we only need to consider its restriction between the two indicatrices.

Let $\{\partial_r$, $\partial_t$ (or $\partial_\theta$ for $F_2$), $X_1,\cdots,X_{n-2}\}$ be any spherical local frame, and denote
$T_1=\partial_t-\tfrac{r}{2f(t)}\tfrac{{\rm d}}{{\rm d}t}f(t)\partial_r$ and
$T_2=\partial_\theta-\tfrac{r}{2h(\theta)}\tfrac{{\rm d}}{{\rm d}\theta}h(\theta)\partial_r$. Similarly, from (\ref{0009}), we see that $\Phi$ preserves the defining domain for this spherical local frame,  $\Phi_*(T_1)=\tfrac{{\rm d}}{{\rm d}t}\theta(t)\  T_2$, and $\Phi_*(X_i)=X_i$, $\forall 1\leq i\leq n-2$, as in Lemma \ref{lemma-6}.
Then
(\ref{ODE-1}) and (\ref{0013}) for all $k\in\{0,\cdots,d-1\}$
implies $$g_1(T_1,T_1)=g_2(\Phi_*(T_1),\Phi_*(T_1)),\quad \mbox{and}\quad
g_1(X_i,X_i)=g_2(\Phi_*(X_i),\Phi_*(X_i)),\ \forall i.$$
Since $\{T_i,X_1,\cdots,X_{n-2}\}$ is $g_i$-orthogonal, we see
$\Phi$ is an isometry between $(S_{F_i}\backslash(N_{i,0}\cup N_{i, {\pi}/d})$, $g_i)$.

Then we prove that $\Phi$ can be continuously glued with some scalar multiplications from $N_{1,t}$ to $N_{2,t}$ for $t=0,\tfrac{\pi}d$ respectively. Equivalently, we show the following map, $$\Phi_1(x)=\tfrac{\Phi(x)}{|\Phi(x)|}, \ \forall x\in S^{n-1}(1)\backslash (M_0\cup M_{\pi/d}),\quad \Phi_1(x)=x,\ \forall x\in M_0\cup M_{\pi/d},$$
is continuous on $S^{n-1}(1)$.
Using the exponential maps for the normal bundle of the two focal submanifolds $M_0$ and $M_{\pi/d}$ in $(S^{n-1}(1),g^{\mathrm{st}})$, we see $S^{n-1}(1)$ has a topological basis consisting of the following open subsets:
\begin{enumerate}
\item Open subsets $U$ of $S^{n-1}(1)\backslash
(M_0\cup M_{\pi/d})$;
\item $U_{U',c}$ parametrized by the
open subset $U'\subset M_0$
and the real number $c\in(0,\tfrac{\pi}{d})$, which contains
 every point $x\in S^{n-1}(1)$ which satisfies $\mathrm{dist}_{S^{n-1}(1)}(x,M_{0})<c$ and is contained in some normal geodesic for $M_t$ in $(S^{n-1}(1),g^{\mathrm{st}})$ passing $U'$.
\item $U_{U'',c}$ parametrized by the
open subset $U''\subset M_{\pi/d}$
and the real number $c\in(0,\tfrac{\pi}{d})$, which contains
 every point $x\in S^{n-1}(1)$ which satisfies $\mathrm{dist}_{S^{n-1}(1)}(x,M_{\pi/d})<c$ and is contained in some normal geodesic for $M_t$ in $(S^{n-1}(1),g^{\mathrm{st}})$ passing $U''$.
\end{enumerate}

Restricted to $ S^{n-1}(1)\backslash(M_0\cup M_{\pi/d}) $, $\Phi_1$
is a diffeomorphism with the spherical coordinates presentation
$(1,t,\xi)\mapsto (1,\theta(t),\xi)$. So for any open subset $U\in S^{n-1}(1)\backslash(M_0\cup M_{\pi/d})$, $\Phi_1^{-1}(U)$ is still
an open subset in $S^{n-1}(1)\backslash(M_0\cup M_{\pi/d})$. By the spherical coordinates presentation for $\Phi_1$, we see
$\Phi_1^{-1}({U}_{c,U'})={U}_{\theta^{-1}(c),U'}$
and $\Phi_1^{-1}({U}_{c,U''})=
{U}_{c_1,U''}$ with $c_1=\tfrac{\pi}d-\theta^{-1}(\tfrac{\pi}d-c)$, for any $c\in(0,\tfrac{\pi}{d})$, any open subset $U'\subset M_0$ and any
open subset $U''\subset M_{\pi/d}$. Here $t$ and $\theta(t)$ are viewed as numbers in $(0,\tfrac{\pi}d)$.
So $\Phi_1$ is continuous on $S^{n-1}(1)$. Meanwhile, we see $\Phi$ can be continuously extended to a homomorphism from $S_{F_1}$ to $S_{F_2}$.

Because $\Phi$ is isometric when restricted to
$(S_{F_i}\backslash(N_{i,0}\cup N_{i,\pi/d}),g_i)$, it is still an isometry after the extension. Finally, by Lemma \ref{lemma-4}, we can use the positive 1-homogeneity to further extend $\Phi$ to a Hessian isometry between $F_1$ and $F_2$.

Above discussion tells us the triple $(F_1,F_2,\Phi)$ can be constructed from the triple $(f(t),h(\theta)$, $\theta(t))$. So we have
the one-to-one correspondence in the following theorem.

\begin{theorem}\label{main-thm-4}
Let $M_t$ be an isoparametric foliation on $(S^{n-1}(1),g^{\mathrm{st}})$ with $d$ principal curvatures.
Then the spherical coordinates presentations, $F_1=r\sqrt{2f(t)}$, $F_2=r\sqrt{2h(\theta)}$ and $\Phi:(r,t,\xi)\mapsto(\tfrac{rf(t)^{1/2}}{h(\theta(t))^{1/2}},\theta(t),\xi)$,
provide the one-to-one correspondence between the set of all triples $(F_1,F_2,\Phi)$ satisfying the following:
\begin{enumerate}
\item $F_1$ and $F_2$ are
Minkowski norms on $\mathbb{R}^n$ induced by $M_t$;
\item $\Phi$ is a Hessian isometry from $F_1$ to $F_2$ which
preserves the orientation and fixes the spherical $\xi$-coordinates.
\end{enumerate}
and the set of all triples $(f(t),h(\theta),\theta(t))$ satisfying the following:
\begin{enumerate}
\item $f(t)$ and $h(\theta)$ are $D_{2d}$-invariant positive smooth functions on $\mathbb{R}/(2\mathbb{Z}\pi)$ satisfying the inequalities
    \begin{eqnarray}
    & &2f(t)\tfrac{{\rm d}^2}{{\rm d}t^2}f(t)-\left(\tfrac{{\rm d}}{{\rm d}t}f(t)\right)^2+4f(t)^2>0,\quad\mbox{and}
    \label{ODIE-1-main-thm-4}\\
    & &2h(\theta)\tfrac{{\rm d}^2}{{\rm d}\theta^2}h(\theta)-\left(\tfrac{{\rm d}}{{\rm d}\theta}h(\theta)\right)^2+4h(\theta)^2>0;
    \label{ODIE-2-main-thm-4}
    \end{eqnarray}
\item $\theta(t)$ is a $D_{2d}$-equivariant orientation preserving
diffeomorphism on $\mathbb{R}/(2\mathbb{Z}\pi)$ which fixes
each point in $\tfrac{\mathbb{Z}\pi}d$;
\item The triple $(f(t),h(\theta),\theta(t))$ is a solution of the following ODE system for all $t\in\mathbb{R}/(2\mathbb{Z}\pi)$,
\begin{eqnarray}
& &\tfrac{1}{2f(t)}\tfrac{{\rm d}^2}{{\rm d}t^2}f(t)-\tfrac{1}{4f(t)^2}\left(\tfrac{{\rm d}}{{\rm d}t}f(t)\right)^2+1\nonumber\\
&=&\left(\tfrac{{\rm d}}{{\rm d}t}\theta(t)\right)^2\
\left(\tfrac{1}{2h(\theta(t))}\tfrac{{\rm d}^2}{{\rm d}\theta^2}h(\theta(t))-\tfrac{1}{4h(\theta(t))^2}\left(\tfrac{{\rm d}}{{\rm d}\theta}h(\theta(t))\right)^2+1\right),\quad\mbox{and}
\label{ODE-1-main-thm-4}\\
   & &\sin^2 (t+\tfrac{k\pi}{d})+\tfrac{\cos(t+\tfrac{k\pi}{d})
\sin(t+\tfrac{k\pi}{d})}{2f(t)}\tfrac{{\rm d}}{{\rm d}t}f(t)\nonumber\\
&=&\sin^2(\theta(t)+\tfrac{k\pi}{d})+
\tfrac{\cos(\theta(t)+\tfrac{k\pi}{d})
\sin(\theta(t)+\tfrac{k\pi}{d})}{2h(\theta(t))}
\tfrac{{\rm d}}{{\rm d}\theta}h(\theta(t))
\label{ODE-2-main-thm-4}
\end{eqnarray}
for each $k\in\{0,\cdots,d-1\}$.
\end{enumerate}
\end{theorem}

\begin{remark} The way we put Theorem \ref{main-thm-4} is explicit and convenient. However, it contains some iterance. For example, when $d\in\{1,2,3\}$, we only need to keep the ODE with $k=0$ for (\ref{ODE-2-main-thm-4}). When $d=2$, the two ODEs
in (\ref{ODE-2-main-thm-4}) are equivalent, because their sum is
$1=1$. When $d=3$, we can use (\ref{ODE-2-main-thm-4}) with $k=0$ and the $\mathbb{Z}_{d}$-symmetry
in $D_{2d}$ to deduce the other ODEs in (\ref{ODIE-2-main-thm-4}). Similarly, when $d=4$ and $6$, we only need the two ODEs with $k=0,1$ for (\ref{ODE-2-main-thm-4}).
\end{remark}
\label{subsection-5-2}
\subsection{Geometric description by (d)-property}
\label{subsection-5-3}
In this subsection, we study the correspondence between the triple
$(F_1,F_2,\Phi)$ in Theorem \ref{main-thm-4} and its restriction
to a normal plane $\mathbf{V}$ for $M_t$. Recall that
$\overline{F}_1=F_1|_{\mathbf{V}}$ and
$\overline{F}_2=F_2|_{\mathbf{V}}$ are two $D_{2d}$-invariant
Minkowski norms on $\mathbf{V}$ with polar coordinates presentations $\overline{F}_1=r\sqrt{2f(t)}$ and $\overline{F}_2=r
\sqrt{2h(\theta)}$. In Section \ref{subsection-5-2}, we have seen the restriction $\overline{\Phi}=
\Phi|_{\mathbf{V}}$ is a Hessian isometry from $\overline{F}_1$
and $\overline{F}_2$ with the polar coordinates presentation
$(r,t)\rightarrow(\tfrac{rf(t)^{1/2}}{h(\theta(t))^{1/2}},\theta(t))$.
Here $(f(t),h(\theta),\theta(t))$ is just the triple in Theorem \ref{main-thm-4} corresponding to $(F_1,F_2,\Phi)$.
The diffeomorphism $\theta(t)$ on $\mathbb{R}/(2\mathbb{Z}\pi)$ is $D_{2d}$-equivariant and preserves the orientation, so $\overline{\Phi}$ is $D_{2d}$-equivariant and preserves the orientation as well. Further more,
since $\theta(t)$ fixes each point in $\tfrac{\mathbb{Z}\pi}d$, $\overline{\Phi}$ preserves each
ray spanned by the points in $\mathbf{V}\cap(M_0\cup M_{\pi/d})$.

The restriction to $\mathbf{V}$ provides the correspondence from the triple $(F_1,F_2,\Phi)$ to the triple $(\overline{F}_1,\overline{F}_2,\overline{\Phi})$, by which we can explain (\ref{ODIE-1-main-thm-4})-(\ref{ODE-2-main-thm-4}) in Theorem \ref{main-thm-4}.

Theorem \ref{main-thm-1} indicates the two inequalities (\ref{ODIE-1-main-thm-4}) and (\ref{ODIE-2-main-thm-4})  just tell us $\overline{F}_1$ and $\overline{F}_2$ are Minkowski norms.

To explain (\ref{ODE-1-main-thm-4}), we
denote $\overline{T}_1=\partial_t-\tfrac{r}{2f(t)}\tfrac{{\rm d}}{{\rm d}t}f(t)\partial_r$ and $\overline{T}_2=\partial_\theta-\tfrac{r}{2h(\theta)}
\tfrac{{\rm d}}{{\rm d}\theta}h(\theta)\partial_r$ the tangent vector
fields  on $\mathbf{V}\backslash\{0\}$  which generate
$S_{\overline{F}_1}$ and $S_{\overline{F}_2}$ respectively. Here $\partial_r$ and $\partial_t$ (or $\partial_\theta$ for $\overline{F}_2$) correspond to the polar $r$- and $t$-coordinates on
$\mathbf{V}$. By Lemma \ref{lemma-6} and the $D_{2d}$-symmetry,
we get $\overline{\Phi}_*(\overline{T}_1)=
\tfrac{{\rm d}}{{\rm d}t}\theta(t)\
\overline{T}_2$. Then by (\ref{0007}) in Lemma \ref{lemma-7}, the ODE (\ref{ODE-1-main-thm-4}) just tells us that
$g^{\overline{F}_1}_x(\overline{T}_1,\overline{T}_1)
=g^{\overline{F}_2}_{\overline{\Phi}(x)}
(\overline{\Phi}_*(\overline{T}_1),\overline{\Phi}_*(\overline{T}_2))$, i.e.,
$\overline{\Phi}$ is
a Hessian isometry between $\overline{F}_i$.

%Let $\{\partial_r,\partial_t$ (or $\partial_\theta$ for $F_2$), $X_1,\cdots,X_{n-2}\}$ be a spherical local frame which defining domain has a nonempty intersection with $\mathbf{V}$. Then the
%canonical smooth extension of $\partial_t$ or $\partial_\theta$ on $\mathbf{V}\backslash\{0\}$ corresponds to the polar $t$-coordinate.
%Similarly, $T_1=\partial_t-\tfrac{r}{2f(t)}\tfrac{{\rm d}}{{\rm d}t}f(t)\partial_r$ and $T_2=\partial_\theta-\tfrac{r}{2h(\theta)}
%\tfrac{{\rm d}}{{\rm d}\theta}h(\theta)\partial_r$ generate the indicatrices of $\overline{F}_1$ and $\overline{F}_2$ respectively.
%Restricted to $\mathbf{V}\backslash\{0\}$, we still have
%$\overline{\Phi}_*(T_1)=\tfrac{{\rm d}}{{\rm d}}\theta(t)\ T_2$, so
%the ODE (\ref{ODE-1-main-thm-4}) just claims $\overline{\Phi}$
%is an isometry between $\overline{F}_i$.

To explain (\ref{ODE-2-main-thm-4}), we recall that the orthonormal coordinates $(x_1,x_2)$ and polar coordinates $(r,t)$ (or $(r,\theta)$ where $\overline{F}_2$ is concerned) of
$x\in\mathbf{V}\backslash\{0\}$ are related by
$$x=x_1v_1+x_2v_2=r\cos t\ v_1+r\sin t\ v_2,$$
in which $v_1\in \mathbf{V}\cap M_0$ and $v_2$ provide an orthonormal basis on $\mathbf{V}$.

Denote $\overline{E}_1=\tfrac{1}2 \overline{F}_1^2=r^2 f(t)$,
then we have
\begin{eqnarray*}
\tfrac{\partial}{\partial x_1}\overline{E}_1=\sin t\tfrac{\partial}{\partial r}\overline{E}_1
+\tfrac1r\cos t\tfrac{\partial}{\partial t}\overline{E}_1=
 2r\sin t f(t)+r\cos t\tfrac{{\rm d}}{{\rm d}t}f(t).
\end{eqnarray*}
So at $x=x_1v_1+x_2v_2\in S_{\overline{F}_1}$, where $2r^2 f(t)=1$ and $x_2=r\sin t$,
\begin{eqnarray*}
x_2\tfrac{\partial}{\partial x_2}\overline{E}_1=
r\sin t\left(2r\sin t\ f(t)+r\cos t\tfrac{{\rm d}}{{\rm d}t}f(t)\right)
=\sin^2 t+\tfrac{\cos t\sin t}{2f(t)}\tfrac{{\rm d}}{{\rm d}t}f(t),
\end{eqnarray*}
which coincides with the left side of (\ref{ODE-2-main-thm-4}).
The right side of (\ref{ODE-2-main-thm-4}) can be expressed similarly. So (\ref{ODE-2-main-thm-4}) tells us, with respect to  the orthogonal decomposition
$$\mathbf{V}=\mathbf{V}'+\mathbf{V}''
=\mathbb{R}v_1+\mathbb{R}v_2 \mbox{ (i.e., we have } \mathbf{V}'=\mathbb{R}v_1\mbox{ and }\mathbf{V}''=\mathbb{R}v_2\mbox{)},$$
for any $x=x'+x''\in S_{F_1}$ and
$\overline{\Phi}(x)=\overline{x}=
\overline{x}'+\overline{x}''\in S_{F_2}$, the following equality is satisfied,
$$g^{\overline{F}_1}_{x}(x'',x)=
g^{\overline{F}_2}_{\overline{x}}(\overline{x}'',\overline{x}).$$

More generally, we define this property as following.
\begin{definition}\label{defining-(d)-property}
A Hessian isometry $\Phi$ between two Minkowski norms $F_1$ and $F_2$ on $\mathbb{R}^n$ with $n\geq2$ is said to satisfy the (d)-property
 with respect to  the orthogonal decomposition
$\mathbb{R}^n=\mathbf{V}'+\mathbf{V}''$, if for any nonzero $x=x'+x''$ and
$ {\Phi}(x)=\overline{x}=\overline{x}'+\overline{x}'' $, with $x',\overline{x}'\in \mathbf{V}'$ and $x'',\overline{x}''\in \mathbf{V}''$, we always have
$g^{ {F}_1}_{x}(x'',x)=
g^{ {F}_2}_{\overline{x}}(\overline{x}'',\overline{x})$ (or equivalently,
$g^{ {F}_1}_{x}(x',x)=F_1(x)^2-
g^{ {F}_1}_{x}(x'',x)=F_2(\overline{x})^2-
g^{ {F}_2}_{\overline{x}}(\overline{x}'',\overline{x})
=g^{ {F}_2}_{\overline{x}}(\overline{x}',\overline{x})$).
\end{definition}
%$x$ and $\overline{x}$ respectively. Then (\ref{ODE-2-in-thm}) can be
%interpreted as the following equation,
%\begin{equation}\label{0016}
%\langle x'',x\rangle_x^{F'_1}=\langle \overline{x}'',\overline{x}\rangle_{\overline{x}}^{F'_2}.
%\end{equation}

So the ODE (\ref{ODE-1-main-thm-4}) with $k=0$ can be interpreted as the (d)-property of $\overline{\Phi}$
for $\mathbf{V}=\mathbf{V}'+\mathbf{V}''
=\mathbb{R}v_1+\mathbb{R}v_2$.
By a similar argument, (\ref{ODE-2-main-thm-4}) with $0<k\leq d-1$
can be interpreted as the (d)-property of $\overline{\Phi}$
for the decomposition
$$\mathbf{V}=\mathbf{V}'+\mathbf{V}''=
\mathbb{R}(\cos(-\tfrac{k\pi}{d}) v_1+\sin(-\tfrac{k\pi}{d})v_2)
+\mathbb{R}(\cos(\tfrac{\pi}{2}-\tfrac{k\pi}{d}) v_1+
\sin(\tfrac{\pi}{2}-\tfrac{k\pi}{d}) v_2).$$

To summarize, we see that $(F_1,F_2,\Phi)$ in Theorem \ref{main-thm-4} determines
$(\overline{F}_1,\overline{F}_2,\overline{\Phi})$ which satisfies
$D_{2d}$-symmetry and  (d)-properties. Conversely, from any
$(\overline{F}_1,\overline{F}_2,\overline{\Phi})$ with $D_{2d}$-symmetry and  (d)-properties, we can retrieve
the triple $(f(t),h(\theta),\theta(t))$ in Theorem \ref{main-thm-4}
and then use it to target $(F_1,F_2,\Phi)$.

Finally, we can transport  $(\overline{F}_1,\overline{F}_2,\overline{\Phi})$ to $\mathbb{R}^2$,
using the identification in Section \ref{subsection-2-4}, which identifies $v_1$ and $v_2$ to
$e_1=(1,0)$ and $e=(0,1)$ in $\mathbb{R}^2$ respectively. By the $D_{2d}$-symmetry, the triple after translation is irrelevant to the choice of the identification between $\mathbf{V}$ and $\mathbb{R}^2$. It does
not depend on the choice of $\mathbf{V}$ either.
So we have the one-to-one correspondence in the following theorem.

\begin{theorem}\label{main-thm-5}
Let $M_t$ be an isoparametric foliation on $(S^{n-1}(1),g^{\mathrm{st}})$ with $d$ principal curvatures.
Then the restriction to a normal plane $\mathbf{V}$ for $M_t$, and the identification between $\mathbf{V}$ and $\mathbb{R}^2$ provide a one-to-one correspondence between the set of all triples
$(F_1,F_2,\Phi)$ satisfying the following:
\begin{enumerate}
\item $F_1$ and $F_2$ are
Minkowski norms on $\mathbb{R}^n$ induced by $M_t$;
\item $\Phi$ is a Hessian isometry from $F_1$ to $F_2$ which
preserves the orientation and fixes the spherical $\xi$-coordinates.
\end{enumerate}
and the set of all triples $(\overline{F}_1,\overline{F}_2,\overline{\Phi})$
satisfying the following:
\begin{enumerate}
\item $\overline{F}_1$ and $\overline{F}_2$ are two $D_{2d}$-invariant Minkowski norms on $\mathbb{R}^2$;
\item $\overline{\Phi}$ is a $D_{2d}$-equivariant orientation preserving Hessian isometry from $\overline{F}_1$ to
    $\overline{F}_2$, which preserves each ray spanned by $(\cos\tfrac{k\pi}{d},\sin\tfrac{k\pi}{d})$ for $k\in\{0,\cdots,2d-1\}$.
\item For each $k\in\{0,\cdots,d-1\}$, with respect to  the decomposition
$$\mathbb{R}^2=\mathbf{V}'+\mathbf{V}''=
\mathbb{R}(\cos(-\tfrac{k\pi}d),\sin(-\tfrac{k\pi}d))
+\mathbb{R}(\cos(\tfrac{\pi}2-\tfrac{k\pi}d,\sin(\tfrac{\pi}2-
\tfrac{k\pi}{d}))),$$
$\overline{\Phi}$ satisfies the (d)-property, i.e., for any nonzero $x=x'+x''$ and
$\overline{\Phi}(x)=\overline{x}=\overline{x}'+\overline{x}''$, with $x',\overline{x}'\in \mathbf{V}'$ and $x'',\overline{x}''\in \mathbf{V}''$, we always have
$g^{\overline{F}_1}_{x}(x'',x)=
g^{\overline{F}_2}_{\overline{x}}(\overline{x}'',\overline{x})$.
\end{enumerate}
\end{theorem}
\begin{remark} When $d=1$ or $2$, Theorem \ref{main-thm-5}
only requires $\overline{\Phi}$ to satisfy
the (d)-property for the decomposition $$\mathbb{R}^2=\mathbf{V}'+\mathbf{V}''=\mathbb{R}e_1+\mathbb{R}e_2,$$
i.e., when $d=2$, exchanging $\mathbf{V}'$ and $\mathbf{V}''$ does not count the second. On the other hand, when $d>2$, essentially more (d)-properties for $\overline{\Phi}$ are required by Theorem \ref{main-thm-5}. This phenomenon and its consequence will be discussed in the next two sections.
\end{remark}
\section{Legendre transformation and (d)-property}
\subsection{Legendre transformation}
\label{subsection-6-1}
In this paper, the {\it Legendre transformation} of a Minkowski norm $F$ on $\mathbb{R}^n$ with $n\geq2$ is  referred to the following. By the strong convexity of $F$, we have the following orientation preserving diffeomorphism,
\begin{equation}\label{eq:define-LT}
\Phi:\mathbb{R}^n\backslash\{0\}\rightarrow
\mathbb{R}^n\backslash\{0\},\quad
(x_1,\cdots,x_n)\mapsto(\tfrac{\partial}{\partial x_1}E,\cdots,\tfrac{\partial}{\partial x_n}E),
\end{equation}
where $E=\tfrac12F^2$. Obviously, $\Phi$ is (locally) linear if and only if $F$ is (locally) Euclidean.
The  image $\Phi(S_{F})$ is a strongly convex sphere surrounding the origin, so it determines a Minkowski norm $\hat{F}$ of $F$ with $S_{\hat{F}}=\Phi(S_{F})$. We denote $g$ and $\hat{g}$  the Hessian metrics of $F$ and $\hat{F}$ respectively.
It is crucial to know that, with respect to  the coordinates $x=(x_1,\cdots,x_n)\in\mathbb{R}^n\backslash\{0\}$,
the Hessian matrix
of $\hat{F}$ at $\overline{x}=\Phi(x)\in\mathbb{R}^n\backslash\{0\}$ coincides with the inverse matrix $(g^{ij})$ for the Hessian matrix $(g_{ij})$ of $F$ at $x$ (see Proposition 14.8.1 in \cite{BCS2000} or Lemma 3.1.2 in \cite{Sh2001}). This observation, together with the fact
$\Phi_*({\partial_{x_i}})=\sum_{j}g_{ij}{\partial_{x_j}}$, implies
\begin{eqnarray*}
\hat{g}(\Phi_*(\partial_{x_i}),\Phi_*(\partial_{x_j}))=
\sum_{k,l}g_{ik}g_{jl}g^{kl}=g_{ij}=
g_1(\partial_{x_i},\partial_{x_j}),
\end{eqnarray*}
i.e., $\Phi$ is a Hessian isometry \cite{Sc2013}. On the other hand, $\Phi$ has the following involutive property.
If we denote
$\Phi(x)=\overline{x}=(\overline{x}_1,\cdots,\overline{x}_n)=
(\sum_i g_{i1}x_i,\cdots,\sum_i g_{in}x_i)$ for
$x=(x_1,\cdots,x_n)\in\mathbb{R}^n\backslash\{0\}$, then we see
\begin{eqnarray*}
\Phi^{-1}(\overline{x}_1,\cdots,\overline{x}_n)=(x_1,\cdots,x_n)
=(\sum_i g^{i1}\overline{x}_i,\cdots,\sum_i g^{in}\overline{x}_i)
=(\tfrac{\partial}{\partial\overline{x}_1}\hat{E},\cdots,
\tfrac{\partial}{\partial\overline{x}_n}\hat{E}),
\end{eqnarray*}
where $\hat{E}=\tfrac12\hat{F}^2$, i.e., $\Phi^{-1}$ is the Legendre transformation of $\hat{F}$.

To summarize, we call $\hat{F}$ the {\it dual (Minkowski) norm} of $F$ and $\Phi$ the {\it Legendre transformation} of $F$.

Notice that our notion
in (\ref{eq:define-LT}) has implicitly used the standard Euclidean inner product to identify $\mathbb{R}^n$ with its dual. However, it does not depend
on the choice of orthonormal coordinates. So we have the following easy lemma.

\begin{lemma}\label{lemma-9}
If a linear isometry on $(\mathbb{R}^n,F)$ preserves the standard
inner product, then it commutes with the Legendre transformation of
$F$ and preserves the dual norm $\hat{F}$.
\end{lemma}

\subsection{Hessian isometries satisfying all (d)-properties}
\label{subsection-6-2}

From Theorem \ref{main-thm-5}, we have seen the importance of
the (d)-property. Indeed, Legendre transformation is one of its origin.

\begin{lemma}\label{lemma-8}
Let $F_1$ and $F_2$ be two Minkowski norm on $\mathbb{R}^n$ with $n\geq2$, and $\Phi$ a Hessian isometry from $F_1$ to $F_2$.
If $\Phi$ is a positive scalar multiplication or the composition between
the Legendre transformation of $F_1$ and a positive scalar multiplication, then
$\Phi$ satisfies the (d)-property for every orthogonal decomposition.
\end{lemma}

\begin{proof}Let $\mathbb{R}^n=\mathbf{V}'+\mathbf{V}''$ be any
orthogonal decomposition. We prove the Hessian isometry in the lemma satisfies the corresponding (d)-property.

Firstly, we prove the case $F_2(cx)=F_1(x)$ and $\Phi(x)=cx$ for some constant $c>0$. With respect to above decomposition, if we have $x=x'+x''\in\mathbb{R}^n\backslash\{0\}$ then $\overline{x}=\Phi(x)=cx=\overline{x}'+\overline{x}''$ satisfies
$\overline{x}'=cx'$ and $\overline{x}''=cx''$. By the observation
\begin{eqnarray*}
g^{F_2}_{\overline{x}}(\overline{x}'',\overline{x})
&=&\tfrac12\tfrac{{\rm d}}{{\rm d}s}(F_2(\overline{x}+s\overline{x}'')^2)|_{s=0}
=\tfrac12\tfrac{{\rm d}}{{\rm d}s}(F_2(cx+csx'')^2)|_{s=0}\\
&=&\tfrac12\tfrac{{\rm d}}{{\rm d}s}(F_1( x+ sx'')^2)|_{s=0}
=g^{F_1}_{x}(x'',x).
\end{eqnarray*}
The (d)-property is proved.

Nextly, we prove the case that $\Phi$ is the Legendre transformation
of $F_1$.
We may choose the orthonormal coordinates $(x_1,\cdots,x_n)$,
such that $\mathbf{V}'$ and $\mathbf{V}''$ are given by $x_{m+1}=\cdots=x_n=0$ and $x_1=\cdots=x_m=0$ respectively. Then for any nonzero
$x=(x_1,\cdots,x_n)=x'+x''$ and $\Phi(x)=\overline{x}=
(\overline{x}_1,\cdots,\overline{x}_n)=\overline{x}'+\overline{x}''$
with $x',\overline{x}'\in\mathbf{V}'$ and $x'',\overline{x}''\in\mathbf{V}''$, we have
\begin{eqnarray*}
g^{F_2}_{\overline{x}}(\overline{x}'',\overline{x})
&=&\sum_{m+1\leq i\leq n, 1\leq j\leq n}\left(\left(\sum_{k}x_k g_{ki}\right)\  g^{ij}\ \left(\sum_{l}x_l g_{jl}\right)\right)\\
&=&\sum_{m+1\leq i\leq n,1\leq p\leq n}x_p g_{pi}x_i=
g^{F_1}_{x}( x'',x),
\end{eqnarray*}
which proves the (d)-property of $\Phi$.

Finally, we prove the case that  $\Phi$ is the composition between the Legendre transformation of $F_1$ and a positive scalar
multiplication.
The argument is a combination of above two, or one may apply Lemma \ref{lemma-11} below.
\end{proof}

Besides the Legendre transformation, identity maps, and their
compositions with positive scalar multiplications, there exists
many other Hessian isometries which satisfy the (d)-property
for every orthogonal decomposition. Here we propose
a construction.

\begin{example}\label{example-1}
Let $C(U_1)$ and $C(U_2)$ be two conic open subsets in $\mathbb{R}^n$ with $n\geq2$ such that their closures only intersect at the origin. We start with the standard Euclidean norm $F_0$ on $\mathbb{R}_n$, and slightly deform it on
$C(U_1)$ and $C(U_2)$ to get the new Minkowski norm $F_1$.
The second new Minkowski norm $F_2$ is constructed by gluing
$F_1$ on $\mathbb{R}^n\backslash C(U_1)$ and the dual norm of $F_1$ on $C(U_1)$. Then there is a Hessian isometry $\Phi$
from $F_1$ to $F_2$, such that $\Phi$ coincides with
the Legendre transformation of $F_1$ on $C(U_1)$ and the identity map elsewhere. This $\Phi$ satisfies the (d)-property for every orthogonal decomposition.
If $F_1$ is locally non-Euclidean on $C(U_1)$ and
$C(U_2)$, then $\Phi$ is not a positive scalar multiplication or the composition between a Legendre transformation and a positive scalar multiplication.
\end{example}

More examples can be
constructed similarly, which may involve more (even infinitely many)
conic open subsets $C(U_i)$ and different scalar changes.

\subsection{Local (d)-property}
\label{subsection-6-3}
For the convenience of later discussion, we also introduce the local version for (d)-property.
\begin{definition}
Let $F_1$ and $F_2$ be two Minkowski norms on $\mathbb{R}^n$
with $n\geq 2$, $C(U_1)$ a connected conic open subset of $\mathbb{R}^n\backslash\{0\}$, and  $\mathbb{R}^n=\mathbf{V}'+\mathbf{V}''$ an orthogonal decomposition. Then a local Hessian isometry $\Phi$ from $F_1$ to $F_2$ is said to satisfy the local (d)-property on $C(U_1)$ for  $\mathbb{R}^n=\mathbf{V}'+\mathbf{V}''$,
if $\Phi$ has definition on $C(U_1)$, and for any $x=x'+x''\in C(U_1)$ and $\Phi(x)=\overline{x}=\overline{x}_1+\overline{x}_2$, with
$x',\overline{x}'\in\mathbf{V}'$ and $x'',\overline{x}''\in\mathbf{V}''$,
we always have
$g^{\overline{F}_1}_{x}(x'',x)=
g^{\overline{F}_2}_{\overline{x}}(\overline{x}'',\overline{x})$ (or equivalently,
$g^{\overline{F}_1}_{x}(x',x)=F_1(x)^2-
g^{\overline{F}_1}_{x}(x'',x)=F_2(\overline{x})^2-
g^{\overline{F}_2}_{\overline{x}}(\overline{x}'',\overline{x})
=g^{\overline{F}_2}_{\overline{x}}(\overline{x}',\overline{x})$).
\end{definition}

The following transitivity lemma for local (d)-property is easy to see.

\begin{lemma}\label{lemma-11}
Let $F_i$ with $1\leq i\leq 3$ be three Minkowski norms on $\mathbb{R}^n$ with $n\geq2$, $C(U_1)$ and $C(U_2)$ two connected conic open subsets of $\mathbb{R}^n\backslash\{0\}$,
$\Phi_1$ is a local Hessian isometry from $F_1$ to $F_2$, which
maps $C(U_1)$ into $C(U_2)$, $\Phi_2$ is a local Hessian isometry from $F_2$ to $F_3$ which has definition on $C(U_2)$. Then with respect to
the same orthogonal decomposition $\mathbb{R}^n=\mathbf{V}'+\mathbf{V}''$, $\Phi=\Phi_2\circ\Phi_1$
satisfies the (d)-property on $C(U_1)$ when each $\Phi_i$  satisfies
the (d)-property on $C(U_i)$.
\end{lemma}

%On the other hand,  Hessian isometries from $F_1$ to $F_2$ which are of the forms
%$\Phi(x_1,\cdots,x_n)=(ax_1,\cdots,ax_m,bx_{m+1},\cdots,bx_n)$
%and $\Phi(x_1,\cdots,x_n)=(a\tfrac{\partial}{\partial x_1}(\tfrac12 F_1^2),\cdots,a\tfrac{\partial}{\partial x_m}(\tfrac12 F_1^2),b\tfrac{\partial}{\partial x_{m+1}}(\tfrac12 F_1^2),\cdots,b\tfrac{\partial}{\partial x_n}(\tfrac12 F_1^2))$
%satisfy Property A everywhere with respect to  $\mathbb{R}^n=\mathbf{V}'+\mathbf{V}''$ where $\mathbf{V}'$ and
%$\mathbf{V}''$ are given by the equations $x_{m+1}=\cdots=x_n=0$ and
%$x_1=\cdots=x_m=0$ respectively for the orthonormal coordinates $x=(x_1,\cdots,x_n)$.

%However, Property A with respect to  two different orthogonal
%decomposition for $\mathbb{R}^2$ simultaneously imposes
%a strong restriction on local Hessian isometries between Minkowski norms on $\mathbb{R}^2$. Notice that the equivalent between (\ref{0029}) and (\ref{0030}) implies switching $\mathbf{V}'$ and
%$\mathbf{V}''$ in $\mathbb{R}^n=\mathbf{V}'+\mathbf{V}''$ does
%not change the orthogonal decomposition for discussing Property A.
%This thought is implied by the following two lemmas, which are useful in the next section.

We will need
the following lemma in Section \ref{subsection-7-2}.

\begin{lemma}\label{lemma-12}
Let $F_1$ and $F_2$ be two Minkowski norms on
$\mathbb{R}^2$, $C(U_1)$ a connected conic open subset contained
in the first quadrant $\{(x_1,x_2)|x_1>0, x_2>0\}$, and $\Phi$ a local Hessian isometry from $F_1$ to $F_2$. Suppose there exist positive constants $a$ and $b$, such that the local Hessian isometry $\Phi$
from $F_1$ to $F_2$ can be presented either as
$$\Phi(x_1,x_2)=(ax_1,bx_2),\quad\forall x=(x_1,x_2)\in C(U_1)$$
or as
$$\Phi(x_1,x_2)=(a\tfrac{\partial}{\partial x_1}E_1,
b\tfrac{\partial}{\partial x_2}E_1),\quad\forall x=(x_1,x_2)\in C(U_1),$$
where $E_1=\tfrac12F_1^2$.
If $\Phi$ satisfies the local (d)-property on $C(U_1)$ for
the decomposition
$\mathbb{R}^2=\mathbf{V}'+\mathbf{V}''$, in which $\mathbf{V}'$ and $\mathbf{V}''$ are spanned by $(\cos c,\sin c)=\cos c\ e_1+\sin c \ e_2$ and $(-\sin c,\cos c)=-\sin c\ e_1+\cos c\ e_2$ respectively, for some $c\notin \mathbb{Z}\pi/2$,
then $\Phi|_{C(U_1)}$ is either a positive scalar multiplication or the
composition between the Legendre transformation of $F_1$ and a
positive scalar multiplication.
\end{lemma}

\begin{proof} We first prove Lemma \ref{lemma-12} when
$\Phi(x_1,x_2)=(ax_1,bx_2)$, $\forall x=(x_1,x_2)\in C(U_1)$.

Notice that this $\Phi$
satisfies the local (d)-property on $C(U_1)$ for $\mathbb{R}^2=\mathbf{V}'+\mathbf{V}''=
\mathbb{R}e_1+\mathbb{R}e_2$. So we have
\begin{eqnarray}
x_2g^{F_1}_x( e_2,x)=\overline{x}_2 g^{F_2}_{\overline{x}} (e_2,\overline{x})\quad\mbox{and}
\quad
x_1g^{F_1}_x( e_1,x)=\overline{x}_1 g^{F_2}_{\overline{x}} (e_1,\overline{x})
\label{0032}
\end{eqnarray}
for any $x=(x_1,x_2)\in C(U_1)$ and
$\Phi(x)=\overline{x}=(\overline{x}_1,\overline{x}_2)$.

On the other hand, the local (d)-property assumed in the lemma implies
\begin{eqnarray}
& &(-x_1\sin c  +x_2\cos c)g_x^{F_1}
(-\sin c\  e_1+\cos c \  e_2,x)\nonumber\\
&=&
(-\overline{x}_1\sin c+\overline{x}_2\cos c)g_{\overline{x}}^{F_2}
(\sin c\  e_1+\cos c \  e_2,
\overline{x})\label{0033}
\end{eqnarray}
for every $$x=(x_1,x_2)=(x_1\cos c+x_2\sin c)\  (\cos c,\sin c)
+(-x_1\sin c+x_2\cos c)\  (-\sin c,\cos c)\in C(U_1)$$ and
$$\overline{x}=\Phi(x)=(\overline{x}_1,\overline{x}_2)=
(\overline{x}_1\cos c+\overline{x}_2\sin c)\  (\cos c,\sin c)
+(-\overline{x}_1\sin c+\overline{x}_2\cos c)\  (-\sin c,\cos c).$$

Plug in (\ref{0032}), $\overline{x}_1=ax_1$ and
$\overline{x}_2=bx_2$ into (\ref{0033}),
and use the fact that $\sin c \cos c\neq0$ because $c\notin\mathbb{Z}\pi/2$, we get
$$a(a-b)x_1g_{x}^{F_1}( e_2,x)=b(a-b)x_2g_{x}^{F_1}( e_1,x).$$
If $a=b$, there is nothing to prove. Otherwise, we get
\begin{equation}\label{0034}
ax_1g_{x}^{F_1}( e_2,x)=bx_2g_{x}^{F_1}( e_1,x)
\end{equation}
for any $x=(x_1,x_2)\in C(U_1)$.

Now we switch to the polar coordinates presentations $x=(x_1,x_2)=(r\cos t,r\sin t)\in C(U_1)$
and $F_1=r\sqrt{2f(t)}$, then (\ref{0034}) is translated to
the ODE
$$(a\cos^2 t+b\sin^2 t)\tfrac{{\rm d}}{{\rm d}t}f(t)+(b-a)\sin 2t f(t)=0,$$
which can be explicitly solved, i.e., for some positive constant $c'$, $f(t)=c'\left|\tfrac{a+b}{a-b}+\cos 2t\right|$. Then we see $F_1|_{C(U_1)}$ coincides with the Euclidean norm $ \sqrt{\tfrac{4c'}{|a-b|}\left(ax_1^2+bx_2^2\right)}$, i.e.,
$\Phi|_{C(U_1)}$ coincides with the composition between the Legendre transformation of $F_1$ and the scalar multiplication by $\tfrac{|a-b|}{4c'}$.

To summarize, we have proved Lemma \ref{lemma-12} when $\Phi(x_1,x_2)=(ax_1,bx_2)$, $\forall x=(x_1,x_2)\in C(U_1)$.
In particular, we see $F_1|_{C(U_1)}$ is a positive constant multiple
of $\sqrt{ax_1^2+bx_2^2}$ when $a\neq b$.

Then we prove Lemma \ref{lemma-12} when
$\Phi(x_1,x_2)=(a\tfrac{\partial}{\partial x_1}E_1,
b\tfrac{\partial}{\partial x_2}E_2)$ for any
$x=(x_1,x_2)\in C(U_1)$.

Denote $C(U_2)=\{(x_1,x_2)| (ax_1,bx_2)\in C(U_1)\}$ another
connected conic open subset in the the first quadrant, $F_3$ the Minkowski norm determined by $F_3(x,y)=F_1(ax,by)$, $\Phi_1(x,y)=(ax,by)$ the linear isometry from $F_3$ to $F_1$,
and $\Phi_2$ the Legendre transformation of $F_2$. The composition
$\Phi\circ\Phi_1$ coincides with the Legendre transformation of $F_3$
on $C(U_2)$. So by the involutive property of Legendre transformation, $\Phi_2\circ\Phi$ maps $(x_1,x_2)$ to $(\tfrac{x_1}{a},\tfrac{x_2}b)$ for any $x=(x_1,x_2)\in C(U_1)$.
By Lemma \ref{lemma-8} and Lemma \ref{lemma-11},
$\Phi_2\circ\Phi$ satisfies the local (d)-property in
Lemma \ref{lemma-12}. As we have proved above, either $a=b$, i.e.,
$\Phi|_{C(U_1)}$ is the composition between the Legendre transformation of $F_1$ and a positive scalar multiplication, or
$F_1|_{C(U_1)}$ is a positive multiple of $\sqrt{\tfrac{ x_1^2}{a}+\tfrac{x_2^2}{b}}$, i.e., $\Phi|_{C(U_1)}$ is
a positive scalar multiplication.
\end{proof}

%-----------------------------------------------------------------

\subsection{Legendre transformation of a Minkowski norm induced by $M_t$}

Now we consider the Legendre transformation $\Phi$ of the Minkowski norm $F_1=r\sqrt{2f(t)}$ induced by the isoparametric
foliation $M_t$ on $(S^{n-1}(1),g^{\mathrm{st}})$.
We prove

\begin{lemma}
\label{lemma-10} Let $F_1=r\sqrt{2f(t)}$ be a Minkowski norm induced by the isoparametric
foliation $M_t$ on $(S^{n-1}(1),g^{\mathrm{st}})$ with $d$ principal curvatures. Then for its dual norm $F_2$ and the Legendre transformation $\Phi$ from $F_1$ to  $F_2$, we have the following:
\begin{enumerate}
\item $F_2$ is also a Minkowski norm induced by $M_t$;
\item $\Phi$ preserves the orientation and fixes the spherical $\xi$-coordinates;
\item Let $\mathbf{V}$ be any normal plane for $M_t$, then the restriction $\overline{\Phi}=\Phi|_{\mathbf{V}}$ coincides with the Legendre transformation of $\overline{F}_1=F_1|_{\mathbf{V}}$.
\end{enumerate}
\end{lemma}
\begin{proof}
With the normal plane $\mathbf{V}$ for $M_t$ arbitrarily chosen,  we can find the
orthonormal coordinates $(x_1,\cdots,x_n)$ for $\mathbb{R}^n$ such that $\mathbf{V}$ is given by $x_3=\cdots=x_n=0$ and $(1,0,\cdots,0)\in M_0$. Around any point $x\in\mathbf{V}\backslash
(\mathbb{R}_{\geq0}M_0\cup\mathbb{R}_{\geq0}M_{\pi/d})$, we
can define a spherical local frame induced by $M_t$.
The values of $X_1,\cdots,X_{n-2}$ span the same tangent subspace
in $T_x(\mathbb{R}^n\backslash\{0\})$ as those of $\partial_{x_3},\cdots,\partial_{x_{n}}$.
By Lemma \ref{lemma-3} and continuity, we have $$\tfrac{\partial}{\partial x_i}E_1|_{\mathbf{V}\backslash\{0\}}=0,\quad \forall 3\leq i\leq n,$$
where $E_1=\tfrac12F_1^2$.
So we have $\Phi(x_1,x_2,0,\cdots,0)=(\tfrac{\partial}{\partial x_1}E_1,
\tfrac{\partial}{\partial x_2}E_1,0,\cdots,0)$, from which we see that $\Phi$
preserves $\mathbf{V}\backslash\{0\}$ and $\overline{\Phi}=\Phi|_{\mathbf{V}\backslash\{0\}}$ is the Legendre transformation for $\overline{F}_1=F_1|_{\mathbf{V}_1}$.

For different normal planes for $M_t$,
$\overline{F}_1$ corresponds to the same
Minkowski norm on $\mathbb{R}^2$, and
then, so its
dual norm $\overline{F}_2=F_2|_{\mathbf{V}_1}$ and its Legendre transformation $\overline{\Phi}=\Phi|_{\mathbf{V}}$ are irrelevant to the choice of $\mathbf{V}$.
This observation implies $F_2$ is a Minkowski norm
induced by $M_t$, and $\Phi$ preserves the foliation.

Since $\overline{F}_1$ is $D_{2d}$-invariant,
its Legendre transformation $\overline{\Phi}$, is $D_{2d}$-equivariant and orientation preserving, and it preserves the ray spanned by
each point in $\mathbf{V}\cap (S^{n-1}(1)\backslash(M_0\cup M_{\pi/d}))$. Let $(r,t)\mapsto (\tfrac{rf(t)^{1/2}}{h(\theta)^{1/2}},\theta(t))$ be
the polar coordinates presentation for $\overline{\Phi}$.
 Then $\theta$ fixes each point in $\tfrac{\mathbb{Z}\pi}d$ and  preserves
the  interval $(\tfrac{k\pi}{d},\tfrac{(k+1)\pi}{d})$ for each $k\in\{0,\cdots,2d-1\}$. Each interval $(\tfrac{k\pi}{d},\tfrac{(k+1)\pi}{d})$ for $t$
determines a conic open subset of $\mathbf{V}\backslash\{0\}$
which corresponds to a distinct spherical $\xi$-coordinate in $\mathbb{R}^n$.
With $\mathbf{V}$ changing arbitrarily, we see from this observation that $\Phi$ fixes the spherical $\xi$-coordinates. Finally, from the spherical coordinates presentation $(r,t,\xi)\mapsto (\tfrac{rf(t)^{1/2}}{h(\theta(t))^{1/2}},\theta(t),\xi)$ for $\Phi$, which
contains the same $\theta(t)$ but restrict $t$ to $(0,\tfrac{\pi}d)$, we see $\Phi$ preserves the orientation.
\end{proof}

Lemma \ref{lemma-10} indicates that the correspondence in
Theorem \ref{main-thm-5} relates Legendre transformations to Legendre transformations. On the other hand, for any
Legendre transformation $\overline{\Phi}$ between the $D_{2d}$-invariant
Minkowski norms $\overline{F}_1$ on $\mathbb{R}^2$ to its dual $\overline{F}_2$, the (d)-property requirement in Theorem \ref{main-thm-5} is met by Lemma \ref{lemma-9} and Lemma \ref{lemma-8}. So $\overline{\Phi}$ is related to a Hessian isometry $\Phi$ from $F_1$ to $F_2$. It must be the same one as in Lemma \ref{lemma-10}, i.e., the
Legendre transformation of $F_1$, because the two coincide when restricted to each normal plane for $M_t$.

To summarize, we have the following one-to-one correspondence between Legendre transformations, which is not affected by the choice for the normal plane $\mathbf{V}$ for $M_t$ or its identification with $\mathbb{R}^2$.

\begin{theorem}\label{main-thm-7}
Let $M_t$ be an isoparametric foliation on $(S^{n-1}(1),g^{\mathrm{st}})$ with $d$ principal curvatures.
Then the restriction to any normal plane $\mathbf{V}$ for $M_t$ and
the identification between $\mathbf{V}$ and $\mathbb{R}^2$ provide a one-to-one correspondence between the set of Legendre transformations of Minkowski norms on $\mathbb{R}^n$ induced by
$M_t$, and the set of Legendre transformations of $D_{2d}$-invariant
Minkowski norms on $\mathbb{R}^2$.
\end{theorem}

\begin{remark}\label{remark-6-8}
The function $\theta(t)$ in the polar coordinates presentation
$(r,t)\mapsto (\tfrac{rf(t)^{1/2}}{h(\theta)^{1/2}},\theta(t))$ for
the Legendre transformation
$\overline{\Phi}(x_1,x_2)=(\tfrac{\partial}{\partial x_1}E_1,
\tfrac{\partial}{\partial x_2}E_1)$ for a $D_{2d}$-invariant
$F_1=r\sqrt{2f(t)}$, with $E_1=\tfrac12F_1^2=r^2f(t)$, has been calculated (with more generality) in \cite{XM2020}, i.e.,
\begin{equation}\label{0018}
\theta(t)=
\arccos\left(\tfrac{2\cos t f(t)-\sin t \tfrac{{\rm d}}{{\rm d}t}f(t)}{\left(4f(t)^2+
\left(\tfrac{{\rm d}}{{\rm d}t}f(t)\right)^2\right)^{1/2}}\right),
\end{equation}
when $t\in(0,\pi)$. It is a solution of
\begin{equation}\label{0019}
\tfrac{{\rm d}}{{\rm d}t}\theta(t)=\tfrac{
\left(2f(t)\tfrac{{\rm d}^2}{{\rm d}t^2}f(t)-\left(\tfrac{{\rm d}}{{\rm d}t}f(t)\right)^2+4f(t)^2\right)\sin\theta(t)\cos\theta(t)}{
\left(\cos t \tfrac{{\rm d}}{{\rm d}t}f(t)+2\sin t f(t)\right)\left(-\sin t
\tfrac{{\rm d}}{{\rm d}t}f(t)+2\cos t f(t)\right)},
\end{equation}
for $t\in(0,\tfrac{\pi}{d})$.
The general
solutions of (\ref{0019}) have also been given by \cite{XM2020}, i.e.,
\begin{eqnarray}
\theta(t)=
\arccos\left(\tfrac{2\cos t f(t)-\sin t \tfrac{{\rm d}}{{\rm d}t}f(t)}{\left[4(\cos^2 t+\tfrac{b^2}{a^2}\sin^2 t)f(t)^2
+4(\tfrac{b^2}{a^2}-1)\cos t\sin t f(t)\tfrac{{\rm d}}{{\rm d}t}f(t)+
(\sin^2 t+\tfrac{b^2}{a^2}\cos^2 t)\left(\tfrac{{\rm d}}{{\rm d}t}f(t)\right)^2\right]^{1/2}}\right),\nonumber\\
\label{0020}
\end{eqnarray}
which appears in the polar coordinates presentation for the mapping
$$(x_1,x_2)\mapsto (a\tfrac{\partial}{\partial x_1}E_1,b\tfrac{\partial}{\partial x_1}E_1),$$
with positive constants $a $ and $b $.
\end{remark}

\section{ODE method and local description}

\subsection{ODE method}
In this section, we analysis the ODE system for the triple $(f(t),h(\theta),\theta(t))$ consisting of (\ref{ODE-1-main-thm-4})
 and (\ref{ODE-2-main-thm-4}) in Theorem \ref{main-thm-4}, i.e.,
\begin{eqnarray}
& &\tfrac{1}{2f(t)}\tfrac{{\rm d}^2}{{\rm d}t^2}f(t)-\tfrac{1}{4f(t)^2}\left(\tfrac{{\rm d}}{{\rm d}t}f(t)\right)^2+1\nonumber\\
&=&\left(\tfrac{{\rm d}}{{\rm d}t}\theta(t)\right)^2\
\left(\tfrac{1}{2h(\theta(t))}\tfrac{{\rm d}^2}{{\rm d}\theta^2}h(\theta(t))-\tfrac{1}{4h(\theta(t))^2}\left(\tfrac{{\rm d}}{{\rm d}\theta}f(t)\right)^2+1\right),
\label{0021}\\
& & \sin^2 (t+\tfrac{k\pi}{d})+\tfrac{\cos(t+\tfrac{k\pi}{d})
\sin(t+\tfrac{k\pi}{d})}{2f(t)}\tfrac{{\rm d}}{{\rm d}t}f(t) \nonumber\\
&=&\sin^2(\theta(t)+\tfrac{k\pi}{d})+\tfrac{\cos(\theta(t)+\tfrac{k\pi}{d})
\sin(\theta(t)+\tfrac{k\pi}{d})}{2h(\theta(t))}
\tfrac{{\rm d}}{{\rm d}\theta}h(\theta(t)), \ \forall k\in\{0,\cdots,d-1\},\label{0022}
\end{eqnarray}
where $d$ is number of principal curvatures for $M_t$.
Recall that $(f(t),h(\theta),\theta(t))$ is from the spherical coordinates
presentations for
the triple $(F_1,F_2,\Phi)$ in Theorem \ref{main-thm-4}, and it determines
the triple $(\overline{F}_1,\overline{F}_2,\overline{\Phi})$
in Theorem \ref{main-thm-5}.

While discussing this ODE system, we assume $f(t)$ satisfying the requirements in Theorem \ref{main-thm-1}
has been arbitrarily chosen and discuss how to locally determine
$\theta(t)$ and $h(\theta)$.
To avoid some minor technical complexity, we fix some $t_0\in (0,\tfrac{\pi}{d})$ (so we have $\theta_0=\theta(t_0)\in(0,\tfrac{\pi}{d})$ as well), and discuss the ODE system for $t$ in a sufficiently small neighborhood of $t_0\in(0,\tfrac{\pi}{d})$. We assume $d>2$ because the cases
$d=1$ and $d=2$ have been discussed by Theorem 1.4 and Theorem 1.5 in \cite{XM2020}. The ODE method in \cite{XM2020} for discussing (\ref{0021}) and (\ref{0022}) with $k=0$ is still a main ingredient here. We sketch it as following.

Rewrite (\ref{0022}) with $k=0$ as
\begin{eqnarray}\label{0023}
\tfrac{1}{h(\theta(t))}\tfrac{{\rm d}}{{\rm d}\theta}h(\theta(t))
=\left({2\sin^2 t+\tfrac{\cos t\sin t}{f(t)}\tfrac{{\rm d}}{{\rm d}t}f(t)}\right){\csc\theta(t)\sec\theta(t)}-2\tan\theta(t),
\end{eqnarray}
and differentiate it with respect to  $t$, then we get
\begin{eqnarray}\label{0024}
& &\tfrac{{\rm d}}{{\rm d}t}\theta(t)\
\left(\tfrac{1}{h(\theta(t))}
\tfrac{{\rm d}^2}{{\rm d}\theta^2}h(\theta(t))-
\tfrac{1}{h(\theta(t))^2}\left(
\tfrac{{\rm d}}{{\rm d}\theta}h(\theta(t))
\right)^2
\right)\nonumber\\
&=&\tfrac{{\rm d}}{{\rm d}t}\theta(t)\
{\left(2\sin^2 t+\tfrac{\cos t\sin t}{f(t)}
\tfrac{{\rm d}}{{\rm d}t}f(t)\right)
\left(\sec^2\theta(t)-\csc^2\theta(t)\right)}
-2\tfrac{{\rm d}}{{\rm d}t}\theta(t)\ \sec^2\theta(t)\nonumber\\
& &+\left({4\cos t\sin t+\tfrac{\cos^2t-\sin^2t}{f(t)}\tfrac{{\rm d}}{{\rm d}t}f(t)
-\tfrac{\cos t\sin t}{f(t)^2}\left(\tfrac{{\rm d}}{{\rm d}t}f(t)\right)^2
+\tfrac{\cos t\sin t}{f(t)}\tfrac{{\rm d}^2}{{\rm d}t^2}f(t)
}\right)
{\csc\theta(t)\sec\theta(t)}.\nonumber\\
\end{eqnarray}
We plug (\ref{0023}) and (\ref{0024}) into the right side of (\ref{0021}), to erase $h(\theta(t))$ and its derivatives, then we get a formal quadratic equation for $\tfrac{{\rm d}}{{\rm d}t}\theta(t)$,
\begin{equation}\label{0025}
A\left(\tfrac{{\rm d}}{{\rm d}t}\theta(t)\right)^2+B\left(
\tfrac{{\rm d}}{{\rm d}t}\theta(t)\right)+C=0,
\end{equation}
in which
\begin{eqnarray*}
A&=&\frac{\cos t\sin t\left(\cos t \tfrac{{\rm d}}{{\rm d}t}f(t)+2\sin t f(t)\right)\left(\sin t \tfrac{{\rm d}}{{\rm d}t}f(t)-2\cos t f(t)\right)}{2f(t)^2\cos^2\theta(t)\sin^2\theta(t)},\\
B&=&\frac{\tfrac{\cos t\sin t}{f(t)}\tfrac{{\rm d}^2}{{\rm d}t^2}f(t)-\tfrac{\cos t\sin t}{f(t)^2}\left(\tfrac{{\rm d}}{{\rm d}t}f(t)\right)^2+\tfrac{\cos^2t-\sin^2t}{f(t)}\tfrac{{\rm d}}{{\rm d}t}f(t)
+4\cos t\sin t
}{\cos\theta(t)\sin\theta(t)},\\
C&=&-\tfrac{1}{f(t)}\tfrac{{\rm d}^2}{{\rm d}t^2}f(t)
+\tfrac{1}{2f(t)^2}\left(\tfrac{{\rm d}}{{\rm d}t}f(t)\right)^2-2.
\end{eqnarray*}
Since we have assumed $d>2$, for $t\in(0,\tfrac{\pi}d)$, we have
$\theta(t)\in(0,\tfrac{\pi}d)\subset(0,\tfrac{\pi}2)$. So
from (\ref{0019}) in Remark \ref{remark-6-8}, we see that $\left(\cos t \tfrac{{\rm d}}{{\rm d}t}f(t)+2\sin t f(t)\right)
\left(\sin t \tfrac{{\rm d}}{{\rm d}t}f(t)-2\cos t f(t)\right)$ is a positive factor in $A$ for $t\in (0,\tfrac{\pi}{d})$. To summarize, when $d>2$,
the $A$-coefficient in (\ref{0025}) is always nonzero for each $t\in(0,\tfrac{\pi}{d})$.

Direct calculation shows that for each $t$ close to $t_0$, the two solutions of
(\ref{0025}) are
\begin{equation}\label{0026}
\frac{\cos\theta(t)\sin\theta(t)}{\cos t\sin t}\quad\mbox{and}\quad
\tfrac{\left(-2f(t)\tfrac{{\rm {\rm d}^2}}{{\rm d}t^2}f(t)+\left(\tfrac{{\rm d}}{{\rm d}t}f(t)\right)^2-4f(t)^2\right)\cos\theta(t)\sin\theta(t)}{
\left(\cos t\tfrac{{\rm d}}{{\rm d}t}f(t)+2\sin t f(t)\right)\left(\sin t\tfrac{{\rm d}}{{\rm d}t}f(t)-2\cos t f(t)\right)}.
\end{equation}
The discriminant of (\ref{0025}) is
\begin{equation*}\label{discriminant}
B^2-4AC=\left(\frac{\cos t\sin t \tfrac{{\rm d}^2}{{\rm d}t^2}f(t)+(\sin^2 t-\cos^2t)\tfrac{{\rm d}^2}{{\rm d}t}f(t)}{\cos\theta(t)\sin\theta(t)}\right)^2,
\end{equation*}
which vanishes if and only if
$\cos t\sin t \tfrac{{\rm d}^2}{{\rm d}t^2}f(t)+(\sin^2 t-\cos^2t)\tfrac{{\rm d}^2}{{\rm d}t}f(t)=0$.

\subsection{Case-by-case discussion}
\label{subsection-7-2}
There are two generic possibilities.

{\bf Case 1}. We have $\cos t_0\sin t_0 \tfrac{{\rm d}^2}{{\rm d}t^2}f(t)+(\sin^2 t_0-\cos^2t_0)\tfrac{{\rm d}^2}{{\rm d}t}f(t)\neq0$. Then we have (see Lemma 3.7 in \cite{XM2020})
\begin{lemma}\label{lemma-13}
If $\cos t_0\sin t_0 \tfrac{{\rm d}^2}{{\rm d}t^2}f(t)+(\sin^2 t_0-\cos^2t_0)\tfrac{{\rm d}^2}{{\rm d}t}f(t)\neq0$,
then one of the following
two cases must happen:
\begin{enumerate}
\item For all $t$ sufficiently close to $t_0$, we have
\begin{equation}\label{0027}
\tfrac{{\rm d}}{{\rm d}t}\theta(t)=\tfrac{\cos\theta(t)\sin\theta(t)}{\cos t\sin t};
\end{equation}
\item For all $t$ sufficiently close to $t_0$, we have
\begin{equation}\label{0028}
\tfrac{{\rm d}}{{\rm d}t}\theta(t)=\tfrac{\left(-2f(t)\tfrac{{\rm {\rm d}^2}}{{\rm d}t^2}f(t)+\left(\tfrac{{\rm d}}{{\rm d}t}f(t)\right)^2-4f(t)^2\right)\cos\theta(t)\sin\theta(t)}{
\left(\cos t\tfrac{{\rm d}}{{\rm d}t}f(t)+2\sin t f(t)\right)\left(\sin t\tfrac{{\rm d}}{{\rm d}t}f(t)-2\cos t f(t)\right)}.
\end{equation}
\end{enumerate}
\end{lemma}

Each of (\ref{0027}) and (\ref{0028}) has a unique solution for the
initial value problem $\theta_0=\theta(t_0)\in(0,\tfrac{\pi}d)$. Then input this solution $\theta(t)$ into (\ref{0022}) with $k=0$, and use
the initial value condition $h(\theta_0)=h_0>0$,
we can locally determine $h(\theta)$ around $\theta_0$, as well as the Hessian isometry
$\overline{\Phi}:(r,t)\mapsto (\tfrac{rf(t)^{1/2}}{h(\theta(t))^{1/2}},\theta(t))$ from $\overline{F}_1=r\sqrt{2f(t)}$ to
 $\overline{F}_2=r\sqrt{2h(\theta)}$ on $\mathbb{R}^2$.

{\bf Subcase 1.1}. (\ref{0027}) is satisfied for all $t$ sufficiently close to $t_0$.

The ODE (\ref{0027}) has the solution
$$\theta(t)=\arccos\left(\tfrac{a\cos t}{(a^2\cos^2 t+b^2\sin^2 t)^{1/2}}\right)$$
with suitable positive constants $a$ and $b$ to meet all the initial value
requirements. In this case, $\overline{\Phi}$ coincides with the linear isomorphism $(x_1,x_2)\mapsto (ax_1,bx_2)$ when polar the $t$-coordinate is close to $t_0$.

Here comes the speciality of $d>2$. By Theorem \ref{main-thm-5},
$\overline{\Phi}$ satisfies the (d)-property for the orthonormal decomposition
$\mathbb{R}^2=\mathbf{V}'+\mathbf{V}''=
\mathbb{R}(\cos(-\tfrac{\pi}{d}),\sin(-\tfrac{\pi}{d}))+
\mathbb{R}(\cos(\tfrac{\pi}2-\tfrac{\pi}d),
\sin(\tfrac{\pi}2-\tfrac{\pi}d))$.
So by Lemma \ref{lemma-12}, $\overline{\Phi}$ locally coincides with
a positive scalar multiplication or the composition between the Legendre
transformation of $\overline{F}_1$ and a positive scalar multiplication.

{\bf Subcase 1.2}. (\ref{0028}) is satisfied for all $t$ sufficiently close to $t_0$.

The solution $\theta(t)$ for (\ref{0028}) is provided in (\ref{0020}),
in which the positive parameters $a$ and $b$ can be suitably chosen
to meet all initial value requirements. The corresponding $\overline{\Phi}$ is given by
$\overline{\Phi}(x_1,x_2)= (a\tfrac{\partial}{\partial x_1}E_1,b\tfrac{\partial}{\partial x_1}E_1)$ with $E_1=\tfrac12F_1^2$ when the polar
$t$-coordinate of $x=(x_1,x_2)$ is sufficiently close to $t_0$. Using Theorem \ref{main-thm-5} and Lemma \ref{lemma-12} for the speciality of $d>2$
again, we see that
$\overline{\Phi}$ locally coincides with
a positive scalar multiplication or the composition between the Legendre
transformation of $\overline{F}_1$ and a positive scalar multiplication.

{\bf Case 2}. We have $\cos t \sin t  \tfrac{{\rm d}^2}{{\rm d}t^2}f(t)+(\sin^2 t -\cos^2t )\tfrac{{\rm d}^2}{{\rm d}t}f(t)=0$ for every $t$ sufficiently close to $t_0$. In this case, the two ODEs in
Lemma \ref{lemma-13} are the same. So $\tfrac{{\rm d}}{{\rm d}t}\theta(t)=\tfrac{\cos\theta(t)\sin\theta(t)}{\cos t\sin t}$ is satisfied
for all $t$ sufficiently close to $t_0$. By argument similar to that for Subcase 1.1, we see $\overline{\Phi}$ locally coincides with
a positive scalar multiplication or the composition between the Legendre
transformation of $\overline{F}_1$ and a positive scalar multiplication, when
restricted to the conic open subset with polar $t$-coordinates sufficiently close to $t_0$.

%{\bf Case 2}. We have. In this case, $f(t)$ can be solved
%locally around $t_0$, i.e., $f(t)=a+b\cos 2t$.
%\begin{lemma}If $\cos t \sin t  \tfrac{{\rm d}^2}{{\rm d}t^2}f(t)+(\sin^2 t -\cos^2t )\tfrac{{\rm d}^2}{{\rm d}t}f(t)\equiv0$ around $t_0$, then for $t$ sufficiently close to $t_0$, $f(t)=a+b\cos 2t$
%for some constants $a$ and $b$, and $\theta
%\end{lemma}
%\subsection{Local description}

\subsection{Conclusion}
We can translate the case-by-case discussion in Section
 \ref{subsection-7-2} to the following theorem, which provides the local description
for a Hessian isometry which preserves the orientation and fixes the spherical $\xi$-coordinates, between two Minkowski norms induced by
 the same isoparametric foliation $M_t$ on $(S^{n-1}(1),g^{\mathrm{st}})$ with $d>2$ principal curvatures.

\begin{theorem} \label{main-thm-6}
Let $M_t$ be an isoparametric foliation on
$(S^{n-1}(1),g^{\mathrm{st}})$ with $d>2$ principal curvature
values. Then for any Hessian isometry $\Phi$ between two Minkowski norms $F_1$ and $F_2$ on $\mathbb{R}^n$ induced by $M_t$, there is a conic open dense
subset $C(U)$ in $\mathbb{R}^n\backslash\{0\}$, where $U$ is the union of some hypersurfaces in the foliation $M_t$,
such that when restricted to each connected component of $C(U)$,  $\Phi$ coincides either with
a positive scalar multiplication, or with the composition between the Legendre transformation of $F_1$ and a positive scalar multiplication. In particular, with respect to  any orthogonal decomposition $\mathbb{R}^n=\mathbf{V}'+\mathbf{V}''$,
 $\Phi$ satisfies the (d)-property, i.e.,
for any nonzero $x=x'+x''$ and $\Phi(x)=\overline{x}=\overline{x}'+\overline{x}''$,
with $x',\overline{x}'\in\mathbf{V}'$ and $x'',\overline{x}''\in\mathbf{V}''$,
we have $g_x^{F_1}(x'',x)=g_{\overline{x}}^{F_2}(\overline{x}'',\overline{x}),
$.
\end{theorem}
\begin{proof}
Let $(\overline{F}_1,\overline{F}_2,\overline{\Phi})$
be the triple corresponding to $(F_1,F_2,\Phi)$ given by
Theorem \ref{main-thm-5}. The above case by case discussion
based on Theorem \ref{main-thm-4}, Theorem \ref{main-thm-5} and
Lemma \ref{lemma-12}
indicates the existence of an open dense subset $\coprod_{i=1}^\infty(c_i,d_i)\subset (0,\tfrac{\pi}d)$, such that when
$\overline{\Phi}$ is restricted
to each conic open subset $C(U'_i)$ in $\mathbb{R}^2\backslash\{0\}$ determined by the polar coordiantes condition $t\in(c_i,d_i)$, it is either a positive scalar
multiplication or the composition between a Legendre transformation
and a positive scalar multiplication.

Let $C(U_i)$ be the conic open subset
in $\mathbb{R}^n\backslash\{0\}$ determined by the spherical coordinates condition
$t\in(a_i,b_i)$. Then $C(U)=\coprod_{i=1}^\infty C(U_i)$ is a conic dense
open subset in $\mathbb{R}^n\backslash\{0\}$, and each $C(U_i)$ is
a connected component of $C(U)$. When $\overline{\Phi}|_{C(U'_i)}$ is a positive scalar multiplication, obviously
so does $\Phi|_{C(U_i)}$. When $\overline{\Phi}|_{C(U'_i)}$ is the composition between a Legendre transformation and a positive scalar multiplication, so does $\Phi|_{C(U_i)}$ by a local analog of Lemma \ref{lemma-10}.

By Lemma \ref{lemma-8}, we know $\Phi$ satisfies the local (d)-property on each $C(U_i)$ for every orthogonal decomposition of $\mathbb{R}^n$.
Since $C(U)=\coprod_{i=1}^\infty C(U_i)$ is dense in $\mathbb{R}^n\backslash\{0\}$,
by continuity, $\Phi$ satisfies the (d)-property for
 every orthogonal decomposition of $\mathbb{R}^n$.
\end{proof}

\begin{remark}\label{remark-7-3}
The Hessian isometry $\Phi$ in Theorem \ref{main-thm-5}, between two Minkowski norms induced by $M_t$, which preserves the orientation and fixes the spherical $\xi$-coordinates, can be constructed as following. Firstly, we use the similar technique as for Example \ref{example-1} to $D_{2d}$-equivariantly glue positive scalar multiplications and the compositions between Legendre transformations and positive scalar multiplications to construct the triple
$(\overline{F}_1,\overline{F}_2,\overline{\Phi})$, which meets all requirements in Theorem \ref{main-thm-5}. In particular, (d)-properties are satisfied by Lemma \ref{lemma-8}. Then Theorem \ref{main-thm-5} provides the corresponding $(F_1,F_2,\Phi)$, in which $\Phi$ is the wanted Hessian isometry. Finally, the argument for Theorem \ref{main-thm-6} tells us that essentially this is the only construction when we have $d>2$ for the isoparametric foliation $M_t$ on
$(S^{n-1}(1),g^{\mathrm{st}})$.
\end{remark}

Finally, we remark that the local case-by-case discussion in Section 7.2 provide the following description for $\theta(t)$ in Theorem \ref{main-thm-4} when $d>2$.

\begin{theorem} Let $M_t$ be an isoparametric foliation on $(S^{n-1}(1),g^{\mathrm{st}})$ with $d>2$ principal curvatures.
Then for any triple $(f(t),h(\theta),\theta(t))$ in the spherical coordinates presentation in Theorem \ref{main-thm-4}, there exists
an open dense subset $ \coprod_{i=1}^{\infty}(c_i,d_i)$ of $(0,\tfrac{\pi}d)$, such that when restricted to each $(c_i,d_i)$, we
have
$$\mbox{either}\quad\theta(t)\equiv t\quad\mbox{or}\quad
\theta(t)\equiv\arccos\left(\tfrac{
2\cos t f(t)-\sin t\tfrac{{\rm d}}{{\rm d}t}f(t)}{
\left(4f(t)^2+\left(\tfrac{{\rm d}}{{\rm d}t}f(t)\right)^2\right)^{1/2}}\right).$$
\end{theorem}

\noindent
{\bf Acknowledgement}. The author sincerely thank V. Matveev
for precious discussion which inspired this work and supplied the most crucial techniques. He would also thank Zizhou Tang, Jianquan Ge and Wenjiao Yan
for their helpful suggestions. He thanks the reviewers for their precious advices. This paper is
supported by Beijing Natural Science Foundation (No. Z180004),
NSFC (No. 11771331, No. 11821101), Capacity Building for Sci-Tech
Innovation -- Fundamental Scientific Research Funds (No. KM201910028021)

\end{document}